\documentclass[11pt]{article}
\usepackage{graphicx}
\usepackage{amsmath} 
\usepackage{amssymb} 
\usepackage{amsthm}
\usepackage{mathrsfs}
\usepackage{tikz}
\usepackage{tikz-cd}
\usepackage[polutonikogreek,greek,english]{babel} 
\usepackage[top=2cm,bottom=2cm,left=3cm,right=3cm]{geometry}
\usepackage{stmaryrd}
\usepackage{hyperref}
\usepackage{dsfont}
\usepackage{bm}
\usepackage[T1]{fontenc}
\usetikzlibrary{patterns}
\hypersetup{
    colorlinks=true,
    linkcolor=blue,
    filecolor=magenta,      
    urlcolor=cyan,
    pdftitle={Overleaf Example},
    pdfpagemode=FullScreen,
    }

\DeclareMathOperator{\brau}{\rm{Br}}

\newcommand{\eN}{\mathbb{N}}
\newcommand{\eZ}{\mathbb{Z}}

\newcommand{\eQ}{\mathbb{Q}}
\newcommand{\eR}{\mathbb{R}}
\newcommand{\eC}{\mathbb{C}}

\newcommand{\eA}{\mathbb{A}}
\newcommand{\p}{\mathbb{P}}

\newtheorem{tho}{Theorem}[section]
\newtheorem{cor}[tho]{Corollary}
\newtheorem{lemme}[tho]{Lemma}

\newtheorem{exam}[tho]{Example}
\newtheorem{prop}[tho]{Proposition}

\newtheorem{rmk}[tho]{Remark}

\title{Solubility for families of norm equations coming from abelian number fields}
\author{ Mathieu Da Silva }

\begin{document}

\maketitle

\textbf{Abstract :} For $F \in \eZ[s,t]$ a binary quadratic form which is irreducible over $\eQ$, and $L$ an abelian number field with class number $1$, we obtain the order of magnitude for the number of values $F(s,t)$ which are a norm from $L$. Our result relies on the fundamental lemma of sieve theory and on geometry of numbers.

\tableofcontents

  \paragraph{Notation}\begin{itemize}
      \item For $B \geqslant 2$, we write $\log_2 B$ for $\log \log B$.
       \item The letter $p$ will always denote a prime number. For $k \in \eN$ and $\nu \in \eN \cup \{0\}$,  we write $p^\nu \mid \mid k$ if $p^\nu \mid k$ and $p^{\nu + 1} \nmid k$.
 \item Let $L$ be a number field. We denote by $\mathcal{O}_L$ its ring of integers, by $h_L$ its class number, and by $C_L := \eA_L^\times/L^\times$ its idèle class group.
  The letter $\mathfrak{p}$ stands for a prime ideal of $\mathcal{O}_L$. We let $\mathcal{F}_L$ (\textit{resp.} $\mathcal{I}_L$, \textit{resp.} $\mathcal{P}_L$) be the set of fractional ideals (\textit{resp.} integral ideals, \textit{resp.} prime ideals) of $\mathcal{O}_L$. If $\mathfrak{a}$ is an ideal of $\mathcal{O}_L$ we denote by $N_{L / \eQ}(\mathfrak{a}) $ its norm and we introduce the function 
 \[ r_L(k) := \# \{ \mathfrak{a} \in \mathcal{I}_L : N_{L / \eQ}(\mathfrak{a}) = k\}.\]For instance, if $L = \eQ(i)$, the quantity $4 r_L(k)$ equals the number of representations of $k$ as a sum of two squares. We also introduce the set 
 \[ \mathcal{N}_L := \{ N_{L/\eQ} (x) : x \in L\}. \]

\item A variety over a field $K$ is an integral separated scheme of finite type over $K$. For each point $x$ of a scheme $X$, we denote by $\kappa(x)$ its residue field. If $X$ is a $K$-scheme, the notation $\overline{X}$ stands for the base change $X \times_K \overline{K}$ where $\overline{K}$ denotes an algebraic closure of $K$.

\item For $k \in \eN$ and $z > 0$, we introduce the quantity
\[ \omega(k,z) := \#\{ p \mid k : p \leqslant z \}.\]

\item For any integers $k,\ell \in \eN$, we write $\ell \mid k^\infty$ if $p \mid \ell $ implies $p \mid k$. Note that $1 \mid k^\infty$ for any $k \in \eN$.

\item For any complex number $s$, we denote by $\sigma $ its real part and by $\tau$ its imaginary part.

\item A function $f : \eN \to \eC$ is said to be multiplicative if $f(nm)= f(n)f(m)$ whenever $n$ and $m$ are coprime integers. For any multiplicative function $f$, we denote by $D_f$ its Dirichlet series
\[ D_f (s) := \sum_{n \geqslant 1} {f(n) \over n^s}.\]

\item For any irreducible binary form $F \in \eZ[s,t]$, we denote by $\rho_F^-$ the multiplicative function defined by 
\[ \label{tau_F} \rho_F^-(k) := \# \{ \xi \bmod k : F(\xi , 1 ) \equiv 0 \bmod k \}.\]We also introduce 
\[ \label{tau_F0} \rho_F(k) := \# \{ (\xi_1, \xi_2) \bmod k : F(\xi_1, \xi_2 ) \equiv 0 \bmod k\}.\]For any $a,k \in \eN$, we define 
\[ \rho_F^-(k,a) := \prod_{\substack{p^\nu \mid \mid k \\ p \nmid a}} \rho_F^-(p^\nu).\]

\end{itemize}
\section{Introduction}\label{section 1}

\textbf{1.1 Motivation.} Let $L/\eQ$ be an abelian number field of degree $n \geqslant 2$, and $F \in \eZ[s,t]$ be a binary form which is irreducible over $\eQ$. This paper is concerned with the local and global solubility of the norm equation 
\[\label{norm eq} \tag{1.1} Y_{F,L} : \quad N_{L/ \eQ} (\bm x) = F(s,t) \quad \left(\bm x \in \eQ^n, (s,t) \in \eZ^2\right),\]where $N_{L/ \eQ} $ is considered as a form over $\eQ$ after a choice of integral basis $(\omega_1, \dots, \omega_n)$ of $L$. If $x \in L$ has coordinates $\bm x := (x_1, \dots, x_n) \in \eQ^n$ in the basis $(\omega_1, \dots, \omega_n)$, we write $N_{L/\eQ}(\bm x)$ for $N_{L /\eQ}(x)$. More precisely, letting 
\[ \pi_{F,L} : Y_{F,L} \longrightarrow \eA^2_\eQ \]be the projection on $(s,t)$, 
we provide a lower bound for the quantity 
\[ \tag{1.2} \label{NFC0} N_{F,L}(B) := \# \left\{ (s,t) \in \eA^2(\eZ) : \max (|s|,|t|) \leqslant B, \pi_{F,L}^{-1}(\eQ) \neq \emptyset  \right\}, \]when $\deg F = 2$, $\mathcal{O}_L$ is a principal ideal domain, and $Y_{F,L}$ admits a non-trivial solution. If $L $ and $K := \eQ[x] / (F(x,1))$ are Galois over $\eQ$, we also provide an upper bound for the quantity
\[ \tag{1.3} \label{def Nloc2}  N_{F,L}^\mathrm{loc}(B) := \# \left\{ (s,t) \in \eA^2(\eZ) : \max (|s|,|t|) \leqslant B, \pi_{F,L}^{-1}(\eA_\eQ) \neq \emptyset  \right\}, \]where $\eA_\eQ$ denotes the adèle ring of $\eQ$. In particular, since $N_{F,L}^\mathrm{loc}(B)$ is a natural upper bound for $ N_{F,L} (B)$, we get that the lower bound obtained for $ N_{F,L} (B) $ is sharp. Under our assumptions, it yields
\[ \label{estim pr} \tag{1.4}  N_{F,L} (B)  \asymp {B^2 \over (\log B)^{1 - {r \over n}}},\]where $r$ is the number of irreducible factors of $F$ in $L [s,t]$.\\

A more geometric point of view on general norm equations similar to (\ref{norm eq}) has been studied in \cite{Wei} and \cite{CTHS}. In particular, the Brauer group of such varieties is well understood in this setting. Note that the most recent results from Wei \cite{Wei} require $L / \eQ$ to be abelian.\\

Our question seems to fall within the scope of the Loughran--Smeets conjecture \cite{LS}. However, some technical assumptions required in that conjecture are not satisfied here, due to the fact that we are counting integral (and not rational) points on the basis. For instance, the morphism $\pi_{F,L}$ is not proper, so we cannot directly apply \cite[th. 1.5]{LS} to get an upper bound for $N_{F,L}^\mathrm{loc}(B)$. Furthermore, our problem does not seem to reduce to a fibration 
\[ \pi : X \longrightarrow \p^1_\eQ,\]where $X$ is a smooth variety. Indeed, if $n \neq \deg F$, taking the induced map \[ Y_{F,L} \longrightarrow \eA^2_\eQ \smallsetminus \{(0,0)\} \longrightarrow \p^1(\eQ)\]changes the form of the fibres, since we need to ensure that $\pi^{-1} ([s :t])$ is well-defined for $[s:t] \in \p^1(\eQ)$. We get equations of the form 
\[ s^{\deg F} N_{L/\eQ}(\bm x) = F(s,t)\]for which our method fails.\\

Nevertheless, the computation of the $\Delta$-invariant defined by \cite[(3.11)]{LS}, using \cite[th 1.5]{LS} in the case of $\pi_{F,L}$. This reveals that estimate (\ref{estim pr}) agrees with the order of magnitude predicted by conjecture \cite[1.6]{LS}. Our result differs from \cite[(5.6)]{LS} by a factor $\log B$ because we do not get the contribution of the point at infinity. This is proven in §1.2.\\

Our method and result generalise \cite{Sofos}, in which a sharp lower bound is obtained for the quantity $N_{F,L}(B)$ in the case where $L/ \eQ$ is quadratic. In this case, we recover a conic bundle and a proper model can be constructed, so that \cite[th. 1.1]{Sofos} is relevant with regard to conjecture \cite[1.6]{LS}. For technical reasons, the degree of $F$ in \cite[th. 1.1]{Sofos} cannot be strictly greater than $3$. The obstruction to extending this result to $F$ of any degree comes from the size of the error term when using the method of Daniel \cite{Daniel}, which relies on the classical estimates for the lattice counting problem. For the same reasons, allowing $L/\eQ$ to have any degree $n \geqslant 2$ makes our strategy work only in the case $\deg F = 2$.\\

We now state our main result.

\begin{tho}\label{th}
Let $L/\eQ$ be an abelian extension of finite degree $n \geqslant 2$ and let $F \in \eZ[s,t]$ be any irreducible binary quadratic form. Assume that \begin{enumerate}
    \item[(i)] the ring $\mathcal{O}_L$ is a principal ideal domain;
    \item[(ii)] there exist $(s_0,t_0) \in \eZ^2$ and $x \in L$ such that $N_{L/\eQ}(x) = F(s_0,t_0)$ and $F(s_0, t_0)$ is coprime to the conductor of $L$.
\end{enumerate}For $B \geqslant 2$, the quantity $ N_{F,L}(B) $ defined by $($\ref{NFC0}$)$ satisfies
\[ N_{F,L}(B) \asymp   {B^2 \over   (\log B)^{ 1-{r \over n}  }} , \]where $r$ is the number of irreducible factors of $F$ in $L[s,t]$.
 \end{tho}Note that in the case $n = 2$, we recover \cite[th. 1.1]{Sofos}.

\begin{rmk}
    It is still unknown whether there exist infinitely many algebraic number fields of class number $1$ \cite[chap. I. §6. p37]{Neukirch}. However, explicit examples of such number fields can be found in \cite{LMFDB}.
\end{rmk}

\begin{exam}
For instance, take $L := \eQ[x] / (x^3 - 3x -1)$ for which $\mathcal{O}_L$ is a principal ideal domain, and $F(s,t) = s^2 -2t^2$. As an integral basis for $L$, we take $(1,\omega,\omega^2)$ where $\omega$ is a root of $x^3 - 3 x - 1$. We obtain for $\bm x = (x_0,x_1,x_2) \in~\eQ^3$  
\[  N_{L / \eQ}(\bm x) = x_0^3 + x_1^3 + x_2^3 - 3 x_0x_1^2 - 3x_1x_2^2-3 x_0x_2^2 + 6 x_0x_1x_2 \]and Theorem \ref{th} provides

\[ \# \left\{ (s,t) \in [-B,B]^2 : \exists \bm x \in \eQ^3, N_{L/\eQ} (\bm x) = s^2 -2t^2 \right\} \asymp {B^2 \over (\log B)^{2/3}} \]as $B $ goes to $+ \infty$. This set is not empty since $(s,t) = (1,1)$ and $(x_0,x_1,x_2)= (0,1,1)$ are a solution.\\
\end{exam}
  
 We recall that the Hasse norm principle does not hold for abelian extensions in general. However, under the assumptions in Theorem \ref{th}, the two quantities $ N_{F,L} (B)$ and $N^\mathrm{loc}_{F,L} (B)$ have the same order of magnitude.

 \begin{prop}\label{coro}
     Let $L/\eQ$ be a Galois extension of finite degree $n \geqslant 2$ and let $F \in \eZ[s,t]$ be any irreducible binary form such that $K := \eQ[x]/(F(x,1))$ is Galois over $\eQ$. Then, when $B$ goes to $+ \infty$, we have
\[  N^\mathrm{loc}_{F,L} (B) \ll {  B^2 \over  (\log B)^{ 1-{r \over n}  }},   \]where $r$ is the number of irreducible factors of $F$ in $L[s,t]$.
 \end{prop} This result is proven in section \ref{proof coro}.

 \begin{rmk}
     Since $F$ is irreducible over $\eQ$, we have that $\eQ[x]/(F(x,1))$ is Galois if and only if $\eQ[x] /( F(1,x))$ is Galois. Indeed, these two polynomials have the same splitting field since their roots are reciprocals of each other.
 \end{rmk}

 \begin{rmk}
     In particular, under the assumptions of Theorem \ref{th}, we have 
     \[ N^\mathrm{loc}_{F,L} (B) \asymp N_{F,L} (B). \]Following Odoni \cite{Odoni} and Browning--Newton \cite{BN}, we expect that 
     \[ \lim_{B \to + \infty} { N^\mathrm{loc}_{F,L} (B) \over N_{F,L} (B) } = \# (\eQ^\ast \cap N_{L/\eQ}(C_L) )/ N_{L/\eQ}(\eQ^*) > 0\]where $(\eQ^\ast \cap N_{L/\eQ}(C_L) )/ N_{L/\eQ}(\eQ^\ast)$ is called the knot group of $L/\eQ$. This does not prove any kind of Hasse norm principle for $L/ \eQ$, since we make the strong assumption that there exists $(x,(s,t)) \in L \times \eZ^2$ such that $F(s,t) = N_{L/\eQ}(x)$ and $F(s, t)$ is coprime to the conductor of $L$.
 \end{rmk}

\textbf{1.2 Link with the Loughran--Smeets conjecture.} We now explain how the problem of estimating $N_{F,L}(B)$ is linked to \cite[conj. 1.6]{LS}. Assume that $L / \eQ$ is Galois and choose an integral basis $(\omega_1, \dots , \omega_n)$ of $\mathcal{O}_L$. Let $X$ be the variety defined by 
     \[ X : \quad N_{L/\eQ} \left( \sum_{i = 1}^n x_i \omega_i \right)= F(s,t) \quad \subset \eA_\eQ^{n+2}, \]and consider the projection $\pi$ on $(s,t) \in \eA_\eQ^2$. For $B \geqslant 2$, we have \[ N_{F,L}(B) = \# \left\{ (s,t) \in \eA^2(\eZ) : \max(|s|,|t|) \leqslant B : \pi^{-1}(s,t)(\eQ) \neq \emptyset \right\}. \]Since $F$ is irreducible over $\eQ$, the only codimension one point whose fibre is singular is $(F = 0)$ and this fibre is given by  
     \[V : \quad N_{L/\eQ} \left( \sum_{i = 1}^n x_i \omega_i \right)= 0.\]The variety $V$ is projective, and we can compute the invariant $\delta(V)$ as defined in \cite[§3.2]{LS}. In our case, the residue field is $\kappa := \mathrm{Frac} \left(\eQ[s,t]/ (F(s,t)) \right)$ and we consider $\kappa' := \kappa L \subset \overline{\kappa}$, the compositum of $\kappa$ and $L$. The extension $\kappa' / \kappa$ is finite, Galois and satisfies (see \cite[prop. 3.19]{Milne})
     \[ \mathrm{Gal}(\kappa' / \kappa) \simeq \mathrm{Gal}(L/ L \cap \kappa),  \]so it is of degree at most $ [L : \eQ]$, which depends on the number of irreducible factors of $F$ in $L[s,t]$. Keeping the notation from \cite[§3.2]{LS}, we get
     \[ \delta(V) = {1 \over \# \mathrm{Gal}(\kappa'/ \kappa)} \# \left\{ \gamma \in \mathrm{Gal}(\kappa'/ \kappa) : \begin{tabular}{c} $\gamma$ fixes a geometric irreducible component \\
of $V$ of multiplicity $1$ \end{tabular} \right\} .\]We now compute $ \# \mathrm{Gal}(\kappa'/ \kappa)$. Let $K := \eQ[x] / (F(x,1))$. In the field $\kappa$, we have that $\alpha :=s/t$ is algebraic with $F(x,1)$ as minimal polynomial. Hence, $\eQ(\alpha) \simeq K$. Now, writing $s = \alpha t$ in $\kappa$ leads to $\kappa \simeq \eQ(\alpha) (t)$, so that 
\[ \kappa \simeq K(t) \]is the field of rational functions over $K$. Then, the algebraic elements of $\kappa$ are precisely the elements of $K$, so it follows that $L \cap \kappa = L \cap K$ and 
\[ \mathrm{Gal}(\kappa' / \kappa) \simeq \mathrm{Gal}(L/ L \cap K).\]Now, let
\[ f(s) := F(s,1) = \prod_{i = 1}^r f_i(s)  \]be the decomposition of $f$ as a product of irreducible factors $f_i \in L[s]$. \\

Since $L/\eQ$ is Galois, all the polynomials $f_i$ have the same degree $d$. To see this, write $K = \eQ(\beta)$ where $\beta $ is a root of $f$. Then $\beta$ is a root of one of the $f_i$, say $f_1$. Thus, \[L[s]/(f_1) \simeq L(\beta) = LK\]and we get 
\[ \deg f_1 = [L(\beta) : L] = [K : K \cap L]. \]The same argument applies for all the conjugates of $\beta$, which are roots of the polynomials $f_i$, hence the result.\\

Therefore, we get $d r = \deg F = [K: \eQ]$ with $d = [K :K \cap L]$. It follows from the multiplicativity of the degrees that $r = [L \cap K : \eQ]$, so that we finally get
\[  \# \mathrm{Gal}(\kappa'/ \kappa) = [L : L \cap K] = {n \over r}, \label{calc [L: L cap K]}\tag{1.5}\]where $r$ is the number of irreducible factors of $F$ in $L[s,t]$. We thus recover 
\[ \delta(V) = {r \over n} \]since the only element of $\mathrm{Gal}(\kappa' / \kappa) \simeq \mathrm{Gal}(L/ L \cap K)$ which fixes a geometrically irreducible component of $V$ is the identity.\\

\textbf{1.3 Generalisations.} In another direction, Odoni \cite{Odoni} shows that for any algebraic number field $L/ \eQ$, there exist constants $c_L>0$, $\alpha_L \in (0,1)$ depending only on $L$, and a divisor $\gamma_L$ of the narrow class group of $L$ such that
     \[ \label{odoni} \tag{1.6} \# \left\{ m \leqslant B : \exists x \in L,\; N_{L/\eQ}(x) = m \right\} \underset{B \to + \infty}{\sim} \gamma_L^{-1} c_L {B \over (\log B)^{1-\alpha_L}},\]and $\alpha_L ={1 \over [L : \eQ]}$ if $L/\eQ$ is Galois. In particular, if $ L$ is abelian, we recover $\alpha_L = {1 \over n}$ and the exponent that appears in Theorem \ref{th}. This suggests that Theorem \ref{th} should still hold in a more general setting. One could indeed expect that for any algebraic number field $L/\eQ$ and any $F \in \eZ[s,t]$ of degree $d$ irreducible over $\eQ$, if there exist $(s_0,t_0) \in \eZ^2$ and $x \in L$ such that $N_{L/\eQ}(x) = F(s_0,t_0)$ is invertible modulo the conductor of $L$, then
     \[ N_{F,L} (B) \underset{B \to + \infty}{\sim}  \gamma_L^{-1} c_{F,L} {B^d \over (\log B)^{ 1- r \alpha_L}}, \]where $\gamma_L$ is a divisor of the narrow class group of $L$, $c_{F,L} > 0$ depends only on $L$ and $F$, $r$ is the number of irreducible factors of $F$ in $L[s,t]$, and $\alpha_L$ is as in (\ref{odoni}).\\

\textbf{1.4. Detecting the solubility.} Recall that we defined 
\[ r_L(k) := \# \{ \mathfrak{a} \in \mathcal{I}_L : N_{L / \eQ} (\mathfrak{a}) = k \}\]and the set 
 \[ \mathcal{N}_L := \{ N_{L/\eQ} (x) : x \in L\} .\]We also write $G := \mathrm{Gal}(L/\eQ)$ and, since $L/ \eQ$ is abelian (see \cite[chap. V, theorem 1.10 and chap. VII, §11]{Neukirch}), we can consider $q_L$ the conductor of $L/\eQ$, which is defined as 
 \[q_L := \min \left\{ q \in \eN : L \subset \eQ(\zeta_q) \right\},\]where $\zeta_q$ denotes a primitive $q$-th root of unity. In particular, $G$ is a quotient of $(\eZ/ q_L \eZ)^\times$, so every $\chi \in \widehat{G} := \mathrm{Hom}(G,\eC^\ast)$ is identified with a (not necessarily primitive) Dirichlet character modulo $q_L$. For $\chi \in \widehat{G}$, we denote by $q(\chi)$ its conductor, which is a divisor of $q_L$. Finally, we recall the conductor-discriminant formula \cite[chap. VII, 11.9]{Neukirch}
 \[ \mathrm{disc}(L) = \prod_{\chi \in \widehat{G}} q(\chi). \]

\begin{lemme}\label{detecteurs} Let $L/\eQ$ be an abelian extension of degree $n\geqslant 2$. A prime $p$ ramifies in $L$ if and only if $p \mid q_L$, and if $p \nmid q_L$ then we have 
\[ \mathds{1}_{\mathcal{N}_L}(p) = {1 \over n} \sum_{\chi \in \widehat{G}} \chi(p). \]
\end{lemme}
\begin{proof}
    Let $p$ be a prime ramifying in $L$. Since $q_L = \mathrm{lcm}_{\chi \in \widehat{G}}( q(\chi) )$, by the conductor-discriminant formula, this is equivalent to $p \mid \mathrm{disc}(L) $. Hence the primes ramifying in $L$ are exactly the ones dividing~$q_L$. Let $p \nmid q_L$. By \cite[th. 6]{Heilbronn}, we know that $\displaystyle\sum_{\chi \in \widehat{G}}\chi(p) = r_L(p) \in \{0,n\}$ is the number of ideals in $\mathcal{I}_L$ of norm $p$. Since $\mathcal{O}_L$ is principal, an integer $k$ is the norm of an element of $L$ if and only if it is the norm of a fractional ideal. If moreover $k = p$ is prime, it is equivalent to being the norm of an ideal in $\mathcal{I}_L$, hence the conclusion.
\end{proof}

\textbf{1.5. Setting.} Let $L/\eQ$ be an abelian extension of degree $n \geqslant 2$. We start by providing a useful characterisation of the irreducibility of $F$ over $L$. To do this, we recall that global class field theory \cite[prop. 4.3 and th. 5.1]{Tate} yields, for any abelian extensions $L_1 / K_1$, $L_2/ K_2$ with $K_1 \subset K_2$ and $L_1 \subset L_2$, a commutative diagram 
\begin{equation*} \label{corps de classe} \tag{1.7}
\begin{tikzcd}[row sep=huge, column sep = huge]
    C_{K_2} \arrow[r, "\theta_{L_2/K_2}"] \arrow[d, "N_{K_2/K_1}"' ] & \mathrm{Gal}(L_2/ K_2) \arrow[d, "j"] \\
    C_{K_1} \arrow[r, "\theta_{L_1/K_1}"'] & \mathrm{Gal}(L_1/K_1)
  \end{tikzcd}
\end{equation*}where the maps $\theta_{L_i / K_i} $ are surjective, $j : \sigma \in \mathrm{Gal}(L_2/ K_2) \longmapsto \sigma_{\mid L_1} \in \mathrm{Gal}(L_1/K_1)$ is the natural map, and $\ker \theta_{L_i / K_i} = N_{L_i / K_i}(C_{L_i})$. We return to Theorem \ref{th}. If $\chi \in \widehat{G}$, then $\chi \circ \theta_{L/\eQ} $ is a character on $C_\eQ$. If $K$ is a number field, we can thus define $\chi \circ \theta_{L/ \eQ} \circ N_{K / \eQ}$ as a character on~$C_K$, which we denote~$\widetilde{\chi}$. We can now state the following lemma. 

\begin{lemme}\label{réécriture (iii)}
    Let $L/\eQ$ be an abelian extension of degree $n \geqslant 2$, $f \in \eZ[x]$ be a polynomial of degree~$2$ which is irreducible over $\eQ$, and $K := \eQ[x] /(f)$. Then, $f$ is irreducible over $L$ if and only if for all non-trivial $\chi$ in $\widehat{G}$, the character $\widetilde{\chi}$ is non-trivial. Moreover, if $f$ is reducible over $L$, there exists a unique character $\chi \in \widehat{G}$ which is non-trivial and such that $\widetilde{\chi} $ is trivial.
\end{lemme}
\begin{proof}
   We denote by $LK$ the compositum of $L$ and $K$. Note that $f$ is irreducible over $L$ if and only if $\deg f = [LK : L]$. From the isomorphism (see \cite[prop. 3.19]{Milne})
    \[ \mathrm{Gal}(LK/ L) \simeq \mathrm{Gal}(K / K \cap L),\]we get
    \[ [LK : L] = { [K : \eQ] \over [K \cap L : \eQ] } = {\deg f \over [K \cap L : \eQ]}. \]Hence, $f$ is irreducible over $L$ if and only if $K \cap L = \eQ$.\\
    
    We now prove that $K \cap L \neq \eQ$ if and only if there exists $\chi \in \widehat{G}$ such that $\widetilde{\chi} $ is trivial. Assume that $E := K \cap L$ is not $\eQ$. Then we have the following commutative diagram 

\begin{equation*}
\begin{tikzcd}[row sep=huge, column sep = huge]
 C_K \arrow[d, "N_{K/E}"] \arrow[black, bend right]{dd}[black,swap]{N_{K/ \eQ}} & \\
   C_E \arrow[r, "\theta_{L/E}"] \arrow[d, "N_{E/\eQ}"] & \mathrm{Gal}(L/E) \arrow[d, "i"] \\
    C_\eQ \arrow[r, "\theta_{L/ \eQ}"'] & G
  \end{tikzcd}
\end{equation*}where $i$ is the natural inclusion. The character extension lemma ensures that this character extends to a non-trivial character $\chi \in \widehat{G}$ such that $\chi \circ i$ is the trivial character on $\mathrm{Gal}(L /E)$. It follows that 
\[ \widetilde{\chi} =\chi \circ \theta_{L/ \eQ} \circ N_{K / \eQ} = \chi \circ \theta_{L/ \eQ} \circ N_{E / \eQ} \circ N_{K / E} = \mathrm{id}\circ \theta_{L / E} \circ N_{K/ E}\]is the trivial character on $C_K$.\\

Assume that there exists $\chi \in \widehat{G}$ non-trivial such that the character $\widetilde{\chi}$ is trivial. Then, we have $\theta_{L/\eQ} (N_{K/ \eQ}(C_K)) \subset \ker \chi$. Hence, the canonical map $G \longrightarrow G / \ker \chi $ induces a surjective map 
\[G / (\theta_{L/\eQ} \circ N_{K/ \eQ} )(C_K) \longrightarrow G / \ker \chi. \]From the commutative diagram 
\begin{equation*}
\begin{tikzcd}[row sep=huge, column sep = huge]
   C_K \arrow[r, "\theta_{LK/K}"] \arrow[d, "N_{K/ \eQ}"'] & \mathrm{Gal}(LK/ K) \arrow[d, "j"] \\
    C_{\eQ} \arrow[r, "\theta_{L/\eQ}"'] & G \\
  \end{tikzcd}
\end{equation*}where $j$ is the natural map, we deduce that there exists a surjective morphism (see \cite[chap. VI, th. 1.10]{Lang})
\[\mathrm{Gal}(K \cap L / \eQ) \simeq G / \mathrm{Gal}(L / K \cap L) \longrightarrow G / \ker \chi. \]Now, since $\chi$ is non-trivial, $\ker \chi $ is a proper normal subgroup of $G$ so it corresponds to a field $L^\chi := L^{\ker \chi }$ such that $[L^\chi : \eQ] > 1$ and $\ker \chi = \mathrm{Gal}(L/ L^\chi)$. Hence, we have found a surjective map 
\[ \mathrm{Gal}(K \cap L / \eQ) \longrightarrow \mathrm{Gal}(L^\chi/\eQ )\]where $\# \mathrm{Gal}( L^\chi / \eQ) > 1$. Therefore $K \cap L \neq\eQ$ and this proves the first part of the lemma.\\

Assume now that $f$ is reducible over $L$. We know that there exists $\chi \in \widehat{G}$ such that $\chi \circ \theta_{L/\eQ} \circ N_{K/\eQ}$ is trivial.\\

To prove that this character is unique, we show that $K = L^\chi$. Indeed, since $K \subset L$ in this case, we have the following commutative diagram 
\begin{equation*}
\begin{tikzcd}[row sep=huge, column sep = huge]
   C_K \arrow[d, "N_{K/\eQ}"'] \arrow[r, "\theta_{L/K}"] & \mathrm{Gal}(L/ K) \arrow[d, "j"] \\
C_{\eQ} \arrow[r, "\theta_{L/ \eQ}"'] & G\\
  \end{tikzcd}
\end{equation*}where $j$ is now the natural inclusion. Therefore, we have $\chi \circ j \circ \theta_{L/K} = \chi \circ \theta_{L / \eQ} \circ N_{K/\eQ} $ which is trivial by assumption, meaning that \[(j \circ \theta_{L/K}) (C_K) = \mathrm{Gal}(L/K) \subset\ker \chi = \mathrm{Gal}(L/ L^\chi),\]which leads to $ L^\chi \subset K$ via the Galois correspondence. Moreover, since we have a surjective morphism
\[\mathrm{Gal}(K  / \eQ) \longrightarrow \mathrm{Gal}(L^\chi /\eQ)  \]with $\# \mathrm{Gal}(K/\eQ)=  2$, it follows that $[L^\chi : \eQ] \leqslant 2$ and since $\chi$ is non-trivial, we have $K = L^\chi$ as announced.

\end{proof}

Using (\ref{corps de classe}) we identify any non-trivial irreducible representation of $G := \mathrm{Gal} (L/\eQ)$ with a non-trivial Dirichlet character modulo $q_L \in \eN$, where $q_L$ is the conductor of $L/\eQ$. We write 
\[ \widehat{G} = \{ 1 , \chi_1, \dots, \chi_{n-1} \}\]and we define the two functions 
\[ \tag{1.8} \label{psi} \psi_L(k) := ( \chi_1 \ast \cdots \ast \chi_{n-1})(k) \quad \quad (k \in \eN), \]and
\[ \tag{1.9} \label{Psi}\Psi_L(\bm k) := \prod_{\ell = 1}^{n-1} \chi_\ell(k_\ell) \quad \quad \left(\bm k \in \eN^{n-1}\right).\]For any $k \in \eN$, we have

\[ \sum_{\substack{\bm k \in \eN^{n-1} \\ k_1 \cdots k_{n-1} =k}} \Psi_L(\bm k) = \psi_L(k). \] Moreover, \cite[th. 6]{Heilbronn} ensures that, since $G$ is abelian, the Dedekind zeta function of $L/\eQ$ is the product of the $L$-functions $L(s, \chi)$, $\chi \in \widehat{G}$, where $\chi$ is seen as a Dirichlet character. Identifying the coefficients of these Dirichlet series, we have for all $k \in \eN$,  
\[\tag{1.10} \label{r_L convolee} r_L(k) = (\mathds{1} \ast \psi_L)(k) = \sum_{\substack{ \bm k \in \eN^{n-1} \\ k_1 \cdots k_{n-1} \mid k}} \Psi_L(\bm k).  \]

The following result will be useful.

\begin{lemme}\label{TNP}
 Let $F \in \eZ[s,t]$ be an irreducible binary form of degree $2$, and $K := \eQ[x]/(F(x,1))$. There exists a constant $\gamma_K$ depending only on $K$ such that, for $z \geqslant 2$, 

\[ \sum_{ p \leqslant z } {\rho_F^-(p) \over p} = \log_2 z + \gamma_K + O\left( {1 \over \log z} \right). \]If $L/ \eQ$ is an abelian number field, $\psi_L $ is the function defined by $($\ref{psi}$)$ and $F$ is irreducible over~$L$, there exists a constant $a_{F,L}$ such that we have, for $z \geqslant 2$,
\[ \sum_{ p \leqslant  z } {\psi_L(p)  \rho_F^-(p) \over p} = a_{F,L} + O \left( {1 \over \log z} \right). \]If $F$ is reducible over $L$, we have 
\[ \sum_{ p \leqslant  z } {\psi_L(p)  \rho_F^-(p) \over p} = \log_2(z) + O(1) . \]
\end{lemme}
\begin{proof} We start by writing that for all but finitely many primes $p$, we have 
    \[ \rho_F^- (p) = \# \{ \mathfrak{p} \in \mathcal{P}_K : \mathfrak{p} \mid p \}.\]Thus, there exists a constant $c_K$ such that
    \[ \sum_{ p \leqslant z } {\rho_F^-(p) \over p} = \sum_{\substack{\mathfrak{p} \in \mathcal{P}_K \\ N_{K/ \eQ}(\mathfrak{p}) \leqslant  z}} {1 \over N_{K/\eQ}(\mathfrak{p})} + c_K + O \left({1 \over z^{1/2}} \right). \]The first part of the lemma follows from the prime number theorem for the Dedekind zeta function of the number field $K$. For the second part of the lemma, we write similarly 
     \[ \sum_{ p \leqslant  z } {\rho_F^-(p) \psi_L(p) \over p} = \sum_{\substack{\chi \in \widehat{G} \\ \chi \neq 1}} \sum_{\substack{\mathfrak{p} \in \mathcal{P}_K \\ N_{K/ \eQ}(\mathfrak{p}) \leqslant z}} {\widetilde{\chi}(\mathfrak{p}) \over N_{K/\eQ}(\mathfrak{p})} + c_{K,L} + O\left( { 1 \over z^{1/2}} \right) \]where $c_{K,L}$ is some constant depending only on $K$ and $L$. By Lemma \ref{réécriture (iii)}, if $F$ is irreducible over $L$ then all the characters $\widetilde{\chi}$ in the above sum are non-trivial. Thus, the prime number theorem for the $L$-functions $L(s,\widetilde{\chi})$ ensures that for each $\chi \neq 1$, we have some constant $a_\chi$ such that 
     \[ \sum_{\substack{\mathfrak{p} \in \mathcal{P}_K \\ N_{K/ \eQ}(\mathfrak{p}) \leqslant z}} {\widetilde{\chi}(\mathfrak{p}) \over N_{K/\eQ}(\mathfrak{p})} = a_\chi + O \left( {1 \over \log z} \right)\]and the result follows in the case where $F$ is irreducible over $L$. Otherwise, Lemma \ref{réécriture (iii)} ensures that exactly one of the non-trivial characters $\chi$ is such that $\widetilde{\chi}$ is trivial. In this case, the prime number theorem for the Dedekind zeta function on $K$ and for the $L$-functions $L(s,\widetilde{\chi})$ provides the result as above.
\end{proof}

\section{Lower bound I : tools from analytic number theory}\label{section3}

For $L/ \eQ$ and $F \in \eZ[s,t]$ as in Theorem \ref{th}, we let $q_L$ be the conductor of the extension $L/ \eQ$ and $K:=\eQ[x]/(F(x,1))$. We introduce the constant 
\[ \label{def b} \tag{2.1} b_F :=  \sup_{(s,t) \in [-1,1]^2} |F(s,t)|.\]Hensel's lemma ensures that for all $p \nmid \mathrm{disc}(F) F(0,1)$ and $\nu \geqslant 1$, we have $\rho_F^-(p^\nu)~=~\rho_F^-(p)$. We thus let
   \[ \tag{2.2} \label{choice W} W := \prod_{p \leqslant w_0} p^{\max(1, v_p(q_L))}\]with $w_0$ large enough so that $q_L \mid W$ and $p > w_0$ implies $p \nmid \mathrm{disc}(F) F(0,1)$. For $w \in \eN$, we let
     \[ \label{G_{F,L}} \tag{2.3} G_{F,L}(s,w) := \prod_{p > w} \left( 1 + \sum_{\nu \geqslant 1}   { \psi_L(p^\nu) \rho_F^-(p^\nu) \over p^{\nu s}} \right).\]Finally, if $\chi \in \widehat{G}$ we let
     \[ L(s,\widetilde{\chi}) = \prod_{\mathfrak{p} \in \mathcal{P}_{K}} \left( 1 - {\widetilde{\chi} (\mathfrak{p}) \over N_{K / \eQ} (\mathfrak{p})^s} \right)^{-1} \quad \quad (\sigma > 1),\]where $\widetilde{\chi} = \chi  \circ \theta_{L/ \eQ} \circ  N_{K / \eQ} $. We now prove several technical lemmas.

 \begin{lemme}\label{convexity bound}
   Let $L/\eQ$ be an abelian extension of degree $n \geqslant 2$ and $F \in \eZ[s,t]$ be a binary quadratic form which is irreducible over $\eQ$. If $w$ is large enough, we have $G_{F,L}(1,w) \neq 0$ and the following holds. The function $G_{F,L}(s,w)$ admits a holomorphic continuation in the region $\sigma > \tfrac12$, and there exists a function $H$ holomorphic in the region $\sigma > {1 \over 2} $ such that $H(s) \asymp_\varepsilon 1$ whenever $\sigma > {1 \over 2} + \varepsilon$ and 
   \[ G_{F,L}(s,w)  = H(s) \prod_{\chi \neq 1} L(s, \widetilde{\chi } )\quad \quad \left( \tfrac12 < \sigma \leqslant 1, \; \tau \in \eR \right). \] In particular, if $F$ is irreducible over $L$, then for all $\varepsilon > 0$,
     \[ \label{conv bound}\tag{$2.4$} G_{F,L}(s,w) \ll_\varepsilon (1 + |\tau |)^{(n-1)(1 - \sigma) + \varepsilon} \quad \quad \left( \tfrac12 < \sigma \leqslant 1, \; \tau \in \eR \right). \]
 \end{lemme}
 \begin{proof}
     For $\sigma > 1$, for $p > w$ and $\mathfrak{p} \in \mathcal{P}_{K}$ such that $\mathfrak{p} \mid p$, we use the inequality $\psi_L(p^\nu) \rho_F^-(p^\nu) \ll_\varepsilon~p^{\nu \varepsilon}$ to write that the product
     \[  \left( 1 + \sum_{\nu \geqslant 1} { \psi_L(p^\nu)\rho_F^-(p^\nu) \over p^{\nu s}} \right) \prod_{\mathfrak{p} \mid p} \prod_{\substack{\chi \in \widehat{G} \\ \chi \neq 1}}\left( 1 - {\widetilde{\chi} (\mathfrak{p}) \over N_{K / \eQ} (\mathfrak{p})^s} \right)\]is equal to 
     \[ 1 + {\psi_L(p) \rho_F^-(p) \over p^s} - {1 \over p^s}\sum_{\substack{\mathfrak{p} \mid p \\ N_{K/ \eQ}(\mathfrak{p}) = p }} \sum_{\substack{\chi \in \widehat{G} \\ \chi \neq 1}} \widetilde{\chi} (\mathfrak{p}) + O_\varepsilon \left( {1 \over p^{2\sigma- \varepsilon }} \right).\]Since the ideals $\mathfrak{p} $ dividing $p$ correspond to the linear factors of $F(x,1) \bmod p$, it follows that 
     \[{1 \over p^s}\sum_{\substack{\mathfrak{p} \mid p \\ N_{K/ \eQ}(\mathfrak{p}) = p }} \sum_{\substack{\chi \in \widehat{G} \\ \chi \neq 1}} \widetilde{\chi} (\mathfrak{p}) =  {\psi_L(p) \rho_F^-(p) \over p^s}, \]hence
     \[\left( 1 + \sum_{\nu \geqslant 1}  { \psi_L(p^\nu)\rho_F^-(p^\nu) \over p^{\nu s}} \right) \prod_{\mathfrak{p} \mid p} \prod_{\substack{\chi \in \widehat{G} \\ \chi \neq 1}}\left( 1 - {\widetilde{\chi} (\mathfrak{p}) \over N_{K / \eQ} (\mathfrak{p})^s} \right)= 1 + O_\varepsilon\left( {1 \over p^{2\sigma- \varepsilon}} \right).\]This equality reveals that $G_{F,L}(1,w) \neq 0$ and $H(s) := G_{F,L}(s,w) \displaystyle\prod_{\substack{\chi \in \widehat{G} \\ \chi \neq 1}} L( s , \widetilde{\chi})^{-1} $ is holomorphic in the region $\sigma > {1 \over 2}$ and satisfies $H(s) \asymp_\varepsilon 1$ whenever $\sigma > {1 \over 2} + \varepsilon$. The conclusion follows, and in particular (\ref{conv bound}) is deduced from the convexity bound \cite[Eq. (5.20)]{IK} for $L(s, \widetilde{\chi})$ in the case $\chi \neq 1$, which can be applied since Lemma \ref{réécriture (iii)} ensures that each $\widetilde{\chi}$ is non-trivial when $F$ is irreducible over $L$.
 \end{proof}

 \begin{rmk}
     The same result holds if we replace $\psi_L$ by any non-trivial $\chi \in \widehat{G}$, still identified with a non-trivial Dirichlet character.
 \end{rmk}

Before stating other technical lemmas that will be required in the following sections, let us introduce a set of multiplicative functions

\[\label{group U} \tag{2.5} \mathcal{U} = \left\{  u : \eN \to \eR_{>0} : \forall k \in \eN, \; u(k) = \prod_{p \mid k } (1 + h(p)) \text{ where } h(p) \ll_u {1 \over p} \right\}. \]The set $\mathcal{U}$ is a group under point-wise multiplication with identity given by the constant application that is equal to $1$. Any $u \in \mathcal{U}$ satisfies $u(k) \leqslant 2^{\omega(k)}$ $(k \in \eN)$ so for all $\varepsilon > 0 $, we have the estimate
\[ \tag{2.6} \label{estim U}u(k) \ll_\varepsilon k^\varepsilon. \]For $u \in \mathcal{U}$ and $\ell \in \eN$, we recall that $u(k,\ell)$ denotes the quantity 
\[ u(k,\ell) = \prod_{\substack{p \mid k \\ p \nmid \ell}} u(p).\]Note that $u(\cdot, \ell) \in \mathcal{U}$ for any $\ell \in \eN$.\\

For $L,F$ as above, $W$ as in (\ref{choice W}), and $h : \eN \to \eR_{>0}$ any multiplicative function bounded by~$\tau$, we let \[ \label{c_L(h)} \tag{2.7} c_L(h) : = \prod_{p\nmid W} \left(1 +  \sum_{\nu \geqslant 1} { h(p^\nu) \psi_L(p^\nu) \over p^{\nu}} \right) \]and 
    \[\label{uh} \tag{2.8} u_{L}(h)(k)  : = \prod_{\substack{p \mid k \\ p \nmid W}} \left(1 +  \sum_{\nu \geqslant 1} {h(p^\nu) \psi_L(p^\nu) \over p^{\nu}} \right)^{-1}.\]Enlarging $w_0$ (see (\ref{choice W})) if necessary, we have $u_h \in \mathcal{U}$ (this uses that $h$ is bounded by $\tau$). For $h = \rho_F^- v$ with $v \in \mathcal{U}$, we write $u_{F,L}(v)$ for $u_{L}(\rho_F^- v)$ and $c_{F,L}(v)$ for $c_L(\rho_F^- v)$. Thus, we have
    \[ \label{c_{F,L}(v)} \tag{2.9} c_{F,L}(v) : = \prod_{p\nmid W} \left(1 +  \sum_{\nu \geqslant 1} { v(p^\nu) \psi_L(p^\nu) \rho_F^- (p^\nu) \over p^{\nu}} \right),  \]and
    \[\label{u_{F,L}(v)} \tag{2.10} u_{F,L}(v)(k) := \prod_{\substack{p \mid k \\ p \nmid W}} \left(1 +  \sum_{\nu \geqslant 1} {v(p^\nu) \psi_L(p^\nu) \rho_F^-(p^\nu) \over p^{\nu}} \right)^{-1}. \]

\begin{lemme}\label{NT sum : main term}
 Let $L/\eQ$ be an abelian number field of degree $n \geqslant 2$, $F \in \eZ[s,t]$ be a binary quadratic form which is irreducible over $L$, let $W$ and $w_0$ be as in $($\ref{choice W}$)$, and let $h~:~\eN~\to~\eR$ be any multiplicative function such that \begin{enumerate}
     \item[(i)] for all $\varepsilon > 0$, for all $k \in \eN$, $h(k) \ll_\varepsilon k^\varepsilon$;
     \item[(ii)] for all prime $p$, $|h(p) - \rho_F^-(p)| \ll {1 \over p}$.
 \end{enumerate}
Let $\varepsilon \in \left( 0, {1 \over 4n} \right)$. Enlarging $w_0$ if necessary, for any $m$ coprime to $W$ and $y > 0$, we have
   \[ \sum_{\substack{k \leqslant y \\ \gcd(k,mW) = 1} }  {\psi_L(k)h(k) \over k} = c_L(h) u_L(h)(m)  + O_\varepsilon\left( m^\varepsilon y^{ \varepsilon - {1 \over 4n}}\right), \]where $c_L(h) > 0$ and $u_L(h) \in \mathcal{U}$ are as in (\ref{c_L(h)}) and (\ref{uh}).
\end{lemme}
\begin{proof}

    Let $W$ and $w_0$ be as in (\ref{choice W}) with $w_0$ large enough to ensure that Lemma \ref{convexity bound} applies. We denote by $f$ and $g$ the multiplicative functions defined by
    \[ f(k) = \psi_L(k) \rho_F^-(k),\]and
    \[ g(k) = \psi_L(k) h(k) \mathds{1}_{\gcd(k,mW) = 1}\]for $k \in \eN$. By $(i)$, we have $ \max(|f(k) |,|g(k)|)  \ll_\varepsilon k^\varepsilon$ for $k \geqslant 1$ and $\varepsilon > 0$. If $p \nmid mW$ and $s \in \eC$ with $\sigma > 1$, it follows that for $\varepsilon > 0$, we have
    \begin{align*} \left( 1 + \sum_{\nu \geqslant 1}  {g(p^\nu) \over p^{\nu s}}  \right) \!\!\left( 1 + \sum_{\nu \geqslant 1}  {f(p^\nu)  \over p^{\nu s}} \right)^{-1} \!\!\!\!\!\!\!\! = & \; \left( 1 + {g(p) \over p^s} + O_\varepsilon \left( {1 \over p^{2\sigma - \varepsilon}} \right) \right) \left( 1 + {f(p) \over p^s } + O_\varepsilon \left( {1 \over p^{2\sigma - \varepsilon}} \right)  \right)^{-1}.   
    \end{align*} If $p > w_0$, we know that $\rho_F^- (p) \leqslant \deg F=2$ and $|\psi_L(p)| \leqslant n-1$. In order to approximate the last factor, we need to ensure that $|f(p)/ p^s| < 1 $, which is possible on enlarging $w_0$ if necessary in order to have $w_0 > 2 (n-1) $. Therefore, for $p \nmid mW$, we have
    \[ \left( 1 + \sum_{\nu \geqslant 1}  {g(p^\nu) \over p^{\nu s}}  \right) \!\!\left( 1 + \sum_{\nu \geqslant 1}  {f(p^\nu)  \over p^{\nu s}} \right)^{-1} \!\!\!\! = 1 +  {g(p)- f(p) \over p^{s} } + O_\varepsilon \left({1 \over p^{2 \sigma - \varepsilon}  }\right),\]and $(ii)$ yields
    \[ \left( 1 + \sum_{\nu \geqslant 1}{g(p^\nu) \over p^{\nu s}}  \right) \left( 1 + \sum_{\nu \geqslant 1}  {f(p^\nu)  \over p^{\nu s}} \right)^{-1} \!\!\!\!\!\!= 1 + O_\varepsilon\left({ 1 \over p^{\min(2\sigma-\varepsilon,1+\sigma)} }\right).\]In particular, the product
    \[ \Phi_m(s) := \prod_{p \nmid mW} \left( 1+ \sum_{\nu \geqslant 1} {g(p^\nu) \over p^{\nu s} }\right) \left(  1+\sum_{\nu \geqslant 1} {f(p^\nu) \over p^{\nu s} } \right)^{-1}\]converges absolutely in the region $\sigma > 1/2$ and has no zero in this region. By analytic continuation we may thus write that whenever $\sigma > 1/2$, we have
    \[ D_g(s) =   \Phi_m(s) G_{F,L}(s,w_0)\prod_{p \mid  m}\left( 1 + \sum_{\nu \geqslant 1} {f(p^\nu ) \over p^{\nu s} } \right), \]where $G_{F,L}(s,w_0)$ is as in (\ref{G_{F,L}}). Note that $G_{F,L}(s,w_0)$ is convergent for $\sigma > {1 \over 2}$ by Lemma \ref{convexity bound}. We used the assumption $\gcd(m,W) = 1$ to rearrange the products. We also have 
    \[ \Phi_m(s) \prod_{p \mid m} \left(1 + \sum_{\nu \geqslant 1} {f(p^\nu) \over p^{\nu s}} \right)  \ll_\varepsilon m^\varepsilon.\] By Lemma \ref{convexity bound}, since $F$ is irreducible over $L$, we deduce the bound 
    \[ D_g(s) \ll_\varepsilon  m^\varepsilon ( 1 + |\tau |)^{(n-1)(1- \sigma)  + \varepsilon } \quad \quad \left( \tfrac12 < \sigma < 1, \; \tau \in \eR \right).\]Taking $w_0$ large enough ensures that $c(h) > 0$ and $u_h \in \mathcal{U}$. It follows that we have the equality $D_g(1) = \Phi_1(1) G_{F,L}(1,w_0) = c_L(h) u_L(h)(m)$, so that
    \[ \sum_{k \leqslant y} {g(k) \over k}  = c_L(h) u_L(h)(m) - \sum_{k > y} {g(k) \over k}. \]A partial summation provides for $z > y$
    \[ \sum_{y < k \leqslant z}  {g (k) \over k} = {1 \over z} \sum_{k \leqslant z} g(k) - {1 \over y} \sum_{k \leqslant y} g(k) + \int_y^z  \left( \sum_{k \leqslant t} g(k) \right) {\mathrm{d}t \over t^2}. \]We will now prove that 
    \[ \sum_{\substack{k \leqslant y  }} g(k) \ll_\varepsilon m^\varepsilon y^{1 - {1 \over 4n} + \varepsilon},\]which implies the estimate
    \[ \sum_{k > y}  {g (k) \over k} = - {1 \over y} \sum_{k \leqslant y} g(k) + \int_y^{+ \infty}  \left( \sum_{k \leqslant t} g(k) \right) {\mathrm{d}t \over t^2} \ll_\varepsilon m^\varepsilon y^{- {1 \over 4n} +\varepsilon }  \]from which the result follows.
    We let $y \geqslant 1$ be a half integer. The Perron formula \cite[cor. 5.3]{MV} with $\sigma_0 = 1 + 1/ \log y$ and $T = y^{1 \over 2n}$ yields 
    \[ \sum_{\substack{k \leqslant y  }} g(k)  = {1 \over 2i \pi} \int_{\sigma_0 - iT}^{\sigma_0 + iT} D_g(s) {y^s \over s} \mathrm{d}s + O_\varepsilon \left( {y^{1 + \varepsilon} \over T } \right).\]Since $D_g(s)$ has no poles in the rectangle enclosed by $\sigma_0 \pm iT$ and $1/2 + \varepsilon \pm iT$, the residue theorem enables us to reduce the problem to bounding the integrals 
    \[ J_1 = \int_{-T}^T D_g( 1/2 + \varepsilon+it){ y^{1/2 + \varepsilon + it} \over 1/2 + \varepsilon + it} \mathrm{d}t,\]
    \[ J_2 = \int_{1/2 + \varepsilon }^{\sigma_0} D_g( \sigma +iT){ y^{\sigma + iT} \over \sigma + iT} \mathrm{d}\sigma,\]and 
    \[ J_3 = \int_{\sigma_0}^{1/2 + \varepsilon } D_g( \sigma -iT){ y^{\sigma - iT} \over \sigma -iT} \mathrm{d}\sigma.\] Using bound (\ref{conv bound}), we have 
    \[ J_1 \ll m^\varepsilon y^{1/2 + \varepsilon }  \int_{1}^T (1 + t)^{ {n-1 \over 2} + n\varepsilon} \mathrm{d}t \ll_\varepsilon m^\varepsilon y^{1/2 +  \varepsilon}T^{{n+1 \over 2} +n\varepsilon} = m^\varepsilon y^{ {3\over 4} + {1 \over 4n} + \varepsilon } \ll_\varepsilon m^\varepsilon y^{ 1 - {1 \over 4n} + \varepsilon }, \]where we used $ n \geqslant 2$. Moreover, the integrals $J_2$ and $J_3$ are bounded by 
    \[m^\varepsilon \int_{1/2 + \varepsilon }^{\sigma_0} T^{(n-1) (1 - \sigma) + \varepsilon}{ y^{\sigma } \over |\sigma + iT| } \mathrm{d}\sigma   \ll_\varepsilon m^\varepsilon  T^{-1} \int_{1/2 + \varepsilon }^{\sigma_0} y^{\sigma + \left( {1 \over 2} - {1 \over 2n} \right) (1 - \sigma) + {\varepsilon \over 2n} }  \mathrm{d}\sigma.\]Since we have
    \begin{align*}  \int_{1/2 + \varepsilon }^{\sigma_0} y^{\sigma + \left( {1 \over 2} - {1 \over 2n} \right) (1 - \sigma) + {\varepsilon \over 2n} }  \mathrm{d}\sigma & =  y^{{1 \over 2} - {1 \over 2n} + {\varepsilon \over 2n}} \log(y)^{-1}  \left(y^{ {\sigma_0 \over 2} \left( 1 + {1 \over n} \right) } - y^{ \left({1 + 2 \varepsilon \over 4}\right) \left(1 + {1 \over n}  \right) }   \right) \\
    & \ll_\varepsilon  y^{1+\varepsilon}, \end{align*}we get that $J_2$ and $J_3$ are $\ll_\varepsilon m^\varepsilon y^{1 - {1 \over 4n} + \varepsilon}$, hence the conclusion when $y \geqslant 1$. The claim remains true when $y \in (0,1)$ because of the estimate $c_L(h) u_L(h)(m) \ll m^\varepsilon$.
    
\end{proof}

For $W$ as in (\ref{choice W}), $a \in \eN$ and $u \in \mathcal{U}$, we consider $\rho_{F,a}^-(\cdot \; ; u)$ the multiplicative function defined by 
\[\label{def func}\tag{2.11} \rho_{F,a}^-(k;u)  :=  \rho_F^- \left( k,a \right)  u(k,a)  \mathds{1}_{\gcd(k,W) = 1}. \]Furthermore, we introduce
\[\label{vartheta} \tag{2.12} \sigma_k(a) := \sum_{\substack{ \ell \mid (ak)^\infty}} {\psi_L(\ell) \rho_{F,k_1a}^-(\ell;u) \over \ell}  \gcd(k_1\ell,a)     .\]

We introduce 
\[ \tag{2.13} \label{fracS}\mathfrak{S} (y,a,k_1;u):= \sum_{\substack{k \leqslant y \\ \gcd(k,W)= 1 }} {\psi_L(k)  \rho_{F,k_1a}^-(k;u) \over k} \gcd(k_1k,a).\]

\begin{cor}\label{lemme central general}
    Let $L/\eQ$ be an abelian number field of degree $n \geqslant 2$, $F \in \eZ[s,t]$ be a binary quadratic form which is irreducible over $L$, let $W$ and $w_0$ be as in $($\ref{choice W}$)$ and $v \in \mathcal{U}$. Let $\varepsilon \in \left( 0, {1 \over 4n} \right)$. Enlarging $w_0$ if necessary, for all $k_1, a \in \eN$ with $\gcd(k_1a,W) =~1$, and for $y > 0$, we have
        \[ \mathfrak{S}(y,a,k_1;v)  = c_{L,F}( v) u_{L,F}(v)(ak_1) \sigma_{k_1}(a)  + O_\varepsilon \left( a^{1 + \varepsilon} k_1^\varepsilon  y^{ -{1 \over 4n} + \varepsilon} \right),\]where $c_{L,F}(v) > 0$ and $u_{L,F}(v) \in \mathcal{U}$ are as in (\ref{c_{F,L}(v)}) and (\ref{u_{F,L}(v)}).
        
\end{cor}
\begin{proof} Every integer $k$ can be written uniquely $k = \ell k'$ with $\gcd(k', ak_1) = 1$ and $\ell \mid (ak_1)^\infty$. We thus start by writing 
\[\mathfrak{S}(y,a,k_1;u)   = \sum_{\substack{ \ell \mid (ak_1)^\infty}} {\psi_L(\ell)  \over \ell} \gcd(k_1\ell,a)\!\!\!\!\!\!\!  \sum_{\substack{k' \leqslant y / \ell \\ \gcd(k',ak_1W) = 1}}\!\!\!\!\!\!\! {\psi_L(k') \rho_F^-(k') u(k')  \over k'}. \]Since $v \in \mathcal{U}$, the multiplicative function $f : k \longmapsto \rho_F^-(k) v(k)$ satisfies $(i)$ and $(ii)$ from Lemma~\ref{NT sum : main term}. Hence, Lemma \ref{NT sum : main term} yields
    \[ \sum_{\substack{k' \leqslant y/\ell \\ \gcd(k',ak_1W) = 1}} {\psi_L(k') \rho_F^-(k') v(k') \over k'}  = c_{F,L}(v) u_{F,L}(v)(ak_1) + O\left( \left(ak_1 {y \over \ell} \right)^\varepsilon \left( {y \over \ell} \right)^{ -{1 \over 4n} } \right). \]The main term thus follows from the definition of $\sigma_k(a)$. We now write that the error term is 
    \[ \ll_\varepsilon a^{1 + \varepsilon}k_1^\varepsilon  y^{- {1 \over 4 n}+ \varepsilon} \sum_{ \ell \mid (ak_1)^\infty} {\psi_L(\ell) \over \ell^{1 - {1 \over 4n} + \varepsilon}}. \]Using the estimate $\psi_L(\ell) \ll_\varepsilon \ell^\varepsilon$, the conclusion follows from the error being 
    \[ \ll_\varepsilon a^{1 + \varepsilon} k_1^\varepsilon y^{- {1 \over 4 n}+ \varepsilon} \prod_{p \mid ak_1 } \left( 1 - {1 \over p^{1 - {1 \over 4n}}} \right)^{-1}. \]
\end{proof}

Recall that the real number $b_F$ is defined by (\ref{def b}). For $B \geqslant 2$ and $z \leqslant b_FB^2$, we introduce  
\[ \tag{2.14} \label{region}\mathcal{R}(B,z) := \{ (s,t) \in  [-B,B]^2 : z\leqslant |F(s,t)| \}.\]For $\ell > 0$ and $0 < z_1 < z_2$, we let
\[ \label{Delta vol} \tag{2.15} \Delta \mathrm{vol}(B,z_1,z_2):= \mathrm{vol}(  \mathcal{R}(B, z_1)) - \mathrm{vol}(  \mathcal{R}(B, z_2) ).  \]In section \ref{section4}, we will also need an estimate for sums of the form 

\[ \label{frac S vol} \tag{2.16}\mathfrak{S}^{\mathrm{vol}}(y,a,k_1;z;u) := \sum_{\substack{k \leqslant y \\ \gcd(k,W)= 1 }} \!\!\!\!\! {\psi_L(k)\rho_{F,k_1a}^-(k;u) \over k} \gcd(k_1k,a)  \mathrm{vol}(  \mathcal{R}(B, zk_1k )).\]This is done in the following lemma.

\begin{lemme}\label{estim final avec vol}
     Let $L/\eQ$ be an abelian number field of degree $n \geqslant 2$, $F \in \eZ[s,t]$ be a binary quadratic form which is irreducible over $L$, let $W$ and $w_0$ be as in $($\ref{choice W}$)$ and $v \in \mathcal{U}$. Let $\varepsilon \in \left( 0, {1 \over 4n} \right)$. Enlarging $w_0$ if necessary, for $B \geqslant 2$, for all $k_1, a \in \eN$ satisfying $\gcd(k_1a, W) =~1$, for all $y > 0$ and $z > 0$ such that $k_1 z y \leqslant b_F B^2$, we have
    \begin{align*} \mathfrak{S}^{\mathrm{vol}}(y,a,k_1;z;v)   = & \; 4c_{F,L}(v)  u_{F,L}(v)(ak_1)  \sigma_{k_1}(a)B^2   \\
    & + O_\varepsilon \left(a^{1 + \varepsilon}k_1^\varepsilon \left( B (k_1z)^{1/2}+ (B^2 + k_1 z y \log B) y^{ - {1 \over 4n } + \varepsilon } \right) \right),\end{align*}where $c_{L,F}(v) > 0$ and $u_{L,F}(v) \in \mathcal{U}$ are as in (\ref{c_{F,L}(v)}) and (\ref{u_{F,L}(v)}). 
\end{lemme}
\begin{proof}
We recall that $\mathfrak{S}(y,a,k_1;v)$ is defined in (\ref{fracS}). A discrete version of the summation by parts provides

\begin{align*} \mathfrak{S}^{\mathrm{vol}}(y,a,k_1;z;v) = & \; \mathrm{vol}( \mathcal{R}(B,k_1z \lfloor y \rfloor )  ) \mathfrak{S}(\lfloor y \rfloor,a,k_1;v)
+R(y,a,k_1;z;v),\end{align*}where 
\[ R(y,a,k_1;z;v) :=  \sum_{1 \leqslant \ell \leqslant \lfloor y \rfloor-1} \!\!\!\!\!\! \mathfrak{S}(\ell,a,k_1;v) \Delta \mathrm{vol}(B, k_1 z\ell, k_1z (\ell +1)). \]Using Corollary \ref{lemme central general} for $\ell \in \{1, \dots , \lfloor y \rfloor \}$, we have \[ \mathfrak{S}(\ell,a,k_1;v) =c_{F,L}(v) u_{F,L}(v)(ak_1)  \sigma_{k_1}(a)+ O_\varepsilon \left( a^{1+\varepsilon}k_1^\varepsilon \ell^{- {1 \over 4n} + \varepsilon}  \right),\]which yields a telescoping sum in $R(y,a,k_1;z;v) $ and provides
 
\begin{align*}
    R(y,a,k_1;z;v) = &\; c_{F,L}(v) u_{F,L}(v)(ak_1)  \sigma_{k_1}(a) \left( \mathrm{vol}(B, k_1 z) - \mathrm{vol}(B,  k_1z \lfloor y \rfloor) \right) \\
    & + O_\varepsilon\left(a^{1+ \varepsilon} k_1^\varepsilon  \sum_{\ell \leqslant \lfloor y \rfloor -1}  {\Delta \mathrm{vol}(B,k_1z\ell,k_1z(\ell +1)) \over \ell^{{1 \over 4n} - \varepsilon}} \right),
\end{align*}
and 
\[ \mathfrak{S}(\lfloor y \rfloor,a,k_1;v) =c_{F,L}(v) u_{F,L}(v)(ak_1)  \sigma_{k_1}(a)+ O_\varepsilon \left( a^{1+\varepsilon}k_1^\varepsilon  y^{- {1\over 4n} + \varepsilon}  \right).\]Using the trivial bound $ \mathrm{vol}( \mathcal{R}(B,k_1z \lfloor y \rfloor )  ) \ll B^2 $, it follows that
\begin{align*} \mathfrak{S}^{\mathrm{vol}}(y,a,k_1;z;v)  =  & \; c_{F,L}(v) u_{F,L}(v)(ak_1) \sigma_{k_1}(a)  \mathrm{vol}(  \mathcal{R}(B, k_1z )) \\
& + O_\varepsilon \left( a^{1+\varepsilon}k_1^\varepsilon \left( B^2 y^{- {1 \over 4n} + \varepsilon} + \sum_{\ell \leqslant \lfloor y \rfloor -1}  {\Delta \mathrm{vol}(B,k_1z\ell,k_1z(\ell +1)) \over \ell^{{1 \over 4n} - \varepsilon}} \right)  \right). \end{align*}If $F$ has no linear factor in $\eR[s,t]$, then we have $V_\infty := \mathrm{vol}\left( \left\{ (s,t) \in \eR^2 : |F(s,t)| < 1 \right\} \right) \ll 1 $ and we can use the estimates
\[ 4 B^2 - \mathrm{vol}(  \mathcal{R}(B, k_1z )) \leqslant  \mathrm{vol}\left( \left\{ (s,t) \in \eR^2 : |F(s,t)| < k_1z \right\} \right) = V_\infty k_1z \]and the quantity $\Delta \mathrm{vol} (B, z\ell , z(\ell +1))$ is bounded by
\begin{align*}   \mathrm{vol}\left( \left\{ (s,t) \in \eR^2 : |F(s,t)| < (\ell+ 1)z  \right\} \right) - \mathrm{vol}\left( \left\{ (s,t) \in \eR^2 : |F(s,t)| < \ell z \right\} \right) = V_\infty z.  \end{align*}If $F$ has a linear factor in $\eR[s,t]$, we have $F(s,t) = (as + bt)(cs + dt)$ with $a,b,c,d \in \eR^*$. If $ad = bc$, we are reduced to the case where $F(s,t) = \alpha(\beta s + \gamma t)^2$ with $\alpha, \beta , \gamma \in \eR^*$. Hence, the quantity $\mathrm{vol}(\{ (s,t) \in [-B,B]^2 : |F(s,t)| \leqslant z \})$ is the area inside the strip in $[-B,B]^2$ delimited by the lines $as + bt = \pm \sqrt{ z |\alpha|^{-1} }$, so we have 
\[ \mathrm{vol}(\{ (s,t) \in [-B,B]^2 : |F(s,t)| \leqslant z \}) \ll_F z^{1/2} B. \]If $ad - bc \neq 0$, a change of variables provides that $\mathrm{vol}(\{ (s,t) \in [-B,B]^2 : |F(s,t)| \leqslant z \})$ is bounded by
\[  \mathrm{vol}\left(\left\{ (x_1,x_2) \in \eR^2 : \begin{tabular}{c} $ -(|a|+|b|)B\leqslant x_1 \leqslant  (|a|+ |b|)B $ \\ $   -(|c|+|d|)B \leqslant x_2 \leqslant (|c|+ |d|)B $ \\$  |x_1x_2| \leqslant z $ \end{tabular} \right\}\right).\]
 It suffices to study the case $F(s,t) = st$, at the cost of replacing $[-B,B]^2$ by a rectangle $I_1 \times I_2$ where the interval $I_j$ is of the form $[-\alpha_jB,\alpha_j B]$ with $\alpha_j > 0$ depending only on $F$. In that case,
\begin{align*}\mathrm{vol}(\{ (s,t) \in [-B,B]^2 : |F(s,t)| \leqslant z \}) & \ll \mathrm{vol}\left( (s,t) \in I_1 \times I_2 : |s| < 1 \right)+ z \int_1^{\alpha_1 B} {\mathrm{d}s \over s} \\
& \ll_F B + z \log B \end{align*}and $\Delta \mathrm{vol} (B, z\ell , z(\ell +1))$ is bounded (up to a multiplicative constant depending at most on $F$) by the area delimited by the two hyperbolas $s \longmapsto {z \ell \over s}$ and $s \longmapsto {z (\ell + 1) \over s}$ in the rectangle $I_1 \times I_2$. To compute this area, we proceed as in Figure $1$.\\

\begin{figure}[h!]\label{fig 1}
\centering

\def\z{3}
\def\elle{1}
\def\B{3}
\def\aone{1.4}
\def\atwo{1}

\definecolor{colorone}{RGB}{238, 52, 35}
\definecolor{colortwo}{RGB}{0, 148, 181}
\definecolor{colorthree}{RGB}{64, 183, 105}

\begin{tikzpicture}[scale=1.6]
\small

\pgfmathsetmacro{\sA}{\z*\elle/(\atwo*\B)}
\pgfmathsetmacro{\sB}{\z*(\elle+1)/(\atwo*\B)}
\pgfmathsetmacro{\sC}{\aone*\B}
\pgfmathsetmacro{\ymax}{\atwo*\B}

\draw[->, thick] (0,0) -- (5.2,0) node[right] {$s$};
\draw[->, thick] (0,0) -- (0,3.8) node[above] {};

\draw[thick] (0,0) rectangle (\aone*\B,\ymax);

\begin{scope}
\clip (0,0) rectangle (\aone*\B,\ymax);

\fill[colortwo!25]
plot[domain=\sA:\sC, samples=200]
(\x, {\z*(\elle+1)/\x})
--
plot[domain=\sC:\sA, samples=200]
(\x, {\z*\elle/\x})
-- cycle;

\fill[colorone!20, pattern=north east lines, pattern color=colorone]
plot[domain=\sA:\sB, samples=200]
(\x, {\z*\elle/\x})
--
(\sB,0) -- (\sA,0) -- cycle;

\fill[colorthree!20, pattern=north west lines, pattern color=colorthree]
plot[domain=\sB:\sC, samples=200]
(\x, {\z/\x})
--
(\sC,0) -- (\sB,0) -- cycle;

\draw[colortwo, thick, domain=\sA:\sC, samples=200]
plot (\x, {\z*\elle/\x});
\draw[colorone, thick, domain=\sB:\sC, samples=200]
plot (\x, {\z*(\elle+1)/\x});

\end{scope}

\draw[dashed, thick] (\sA,0) -- (\sA,\ymax);
\draw[dashed, thick] (\sB,0) -- (\sB,\ymax);

\draw (0,\ymax) -- (-0.08,\ymax);

\node[left] at (0,\ymax) {$\alpha_2 B$};

\fill (\sA,0) circle (1.2pt);
\fill (\sB,0) circle (1.2pt);

\fill (\sC,0) circle (1.2pt);

\node[below=6pt] at (\sA,0) {${z\ell \over \alpha_2 B}$};
\node[below=6pt] at (\sB,0) {${z(\ell+1)\over\alpha_2 B}$};
\node[below=6pt] at (\sC,0) {$\alpha_1 B$};

\node[colortwo] at (3.9,1.0) {$\frac{z\ell}{s}$};
\node[colorone] at (3.9,2.6) {$\frac{z(\ell+1)}{s}$};

\end{tikzpicture}
\caption{Area between the two hyperbolas $s \longmapsto {z \ell \over s}$ and $s \longmapsto {z (\ell + 1) \over s}$ in $I_1 \times I_2$}
\end{figure}
We thus obtain

\[  \Delta \mathrm{vol} (B, z\ell , z(\ell +1)) = z  - \int_{{z\ell \over \alpha_2 B}}^{{z (\ell+1) \over \alpha_2 B}} {z \ell \over s} \mathrm{d}s + \int_{{z (\ell+1) \over \alpha_2 B}}^{\alpha_1 B} {z \over s} \mathrm{d}s, \]so that~\[ \Delta \mathrm{vol} (B, z\ell , z(\ell +1)) \ll z \left(1 - \ell \log \left(1 + {1 \over \ell } \right) + \log\left( {\alpha_1 \alpha_2B^2 \over z (\ell + 1)} \right)\right) \ll_F   z \log B.\] It follows that we have
\begin{align*} \mathfrak{S}^{\mathrm{vol}}(y,a,k_1;z;v) = & \; 4c_{F,L}(v) u_{F,L}(v)(ak_1) \sigma_{k_1}(a)  B^2 \\ & + O_\varepsilon \left(a^{1+\varepsilon} k_1^\varepsilon \left( B(k_1z)^{1/2} + B^2 y^{- {1 \over 4n} + \varepsilon} + z k_1  y^{1 - {1 \over 4n } + \varepsilon }\log B \right) \right).  \end{align*} The inequality $k_1 z y \leqslant b_FB^2$ ensures that the error term is admissible.
\end{proof}

We now provide $(n-1)$-dimensional versions of the previous results, that will be useful in section \ref{section4}. For $\mathcal{A} \subset \eN^{n-1} $, we introduce

\[ \label{frac S diese} \tag{2.17} \boldsymbol{\mathfrak{S}}^\# ( \mathcal{A} ,a ; v) :=  \sum_{\substack{ 
   \bm k \in \mathcal{A} }}\Psi_L(\bm k){\rho_{F,a}^- (k_1\cdots k_{n-1};v) \over k_1 \cdots k_{n-1}} \gcd(k_1 \cdots k_{n-1},a) \]where $\rho_{F,a}^-(\cdot \; ; v)$ is as in (\ref{def func}). 

\begin{lemme}\label{Bound error multi D 1}  Let $L/\eQ$ be an abelian number field of degree $n \geqslant 2$, $F \in \eZ[s,t]$ be a binary quadratic form which is irreducible over $L$, let $W$ and $w_0$ be as in $($\ref{choice W}$)$ and $v \in \mathcal{U}$. Let $\varepsilon \in \left( 0, {1 \over 8n^2} \right)$. Enlarging $w_0$ if necessary, for $y > 0$, we have
   \[ \boldsymbol{\mathfrak{S}}^\# \left( \eN^{n-1} \smallsetminus [1,y]^{n-1} ,a;  v\right) \ll_\varepsilon  a^{1+ \varepsilon}y^{(2n-3) \varepsilon - {1 \over 4n}}. \]In particular,
   \[ \lim_{y \to + \infty} \boldsymbol{\mathfrak{S}}^\# \left(  [1,y]^{n-1} ,a; v\right)  = c_{L,F}( v) u_{L,F}(v)(a) \sigma_{1}(a) . \]
\end{lemme}
\begin{proof}We start by restricting the sum to the $\bm k$ such that $k_1 \cdots k_{n-1} \leqslant y^{n-1}$. Indeed, we have 
\[ \sum_{\substack{ \bm k \in \eN^{n-1} \\ \bm k \notin [1,y]^{n-1} \\
    k_1 \cdots k_{n-1} > y^{n-1} }} \!\!\!\!\!\! \Psi_L(\bm k){\rho_{F,a}^- (k_1\cdots k_{n-1};v) \over k_1 \cdots k_{n-1}}  \gcd(k_1 \cdots k_{n-1},a)   = \sum_{\substack{k > y^{n-1}  }}  \psi_L( k){\rho_{F,a}^- (k;v) \over k} \gcd(k,a),  \]
and by Corollary \ref{lemme central general} with $k_1 = 1$ this quantity is 
\[ \ll_\varepsilon a^{1+\varepsilon}  y^{- {n-1 \over 4n} + \varepsilon (n-1)}. \]Therefore, we get that $ \boldsymbol{\mathfrak{S}}^\# \left( \eN^{n-1} \smallsetminus [1,y]^{n-1} ,a ; v\right) $ is equal to

    \begin{align*} \sum_{\substack{ \bm k \in \eN^{n-1} \\ \bm k \notin [1,y]^{n-1} \\ 
    k_1 \cdots k_{n-1} \leqslant y^{n-1} }} \!\!\!\!\!\!\!\!\! \Psi_L(\bm k){\rho_{F,a}^- (k_1\cdots k_{n-1};v)  \over k_1 \cdots k_{n-1}}  \gcd(k_1 \cdots k_{n-1},a) + O_\varepsilon \left( a^{1+\varepsilon}  y^{- {n-1 \over 4n} + \varepsilon (n-1)} \right).\end{align*}We apply the inclusion-exclusion principle to
    \[ \eN^{n-1} \smallsetminus [1,y]^{n-1} = \bigcup_{i = 1}^{n-1} \{ \bm k \in \eN^{n-1} : k_i > y \} \]in order to deal with the condition $\bm k \notin [1,y]^{n-1}$. We are led to estimate the sum
    \[ \boldsymbol{\mathfrak{S}}_I'(y,a;v) := \sum_{\substack{ \bm k \in \eN^{n-1} \\
    k_1 \cdots k_{n-1} \leqslant y^{n-1}\\ \forall i \in I, \; k_i > y }} \!\!\!\!\!\!\!\!\! \Psi_L(\bm k){\rho_{F,a}^-( k_1\cdots k_{n-1};v) \over k_1 \cdots k_{n-1}} \gcd(k_1 \cdots k_{n-1} ,a),\]where $I$ is any non-empty subset of $\{1, \dots, n-1\}$. Without loss of generality, we can assume that $n-1 \in I$, so that it suffices to prove that when $k_1, \dots, k_{n-2}$ are fixed, the contribution coming from the condition $k_{n-1} > y$ is small enough to ensure that the expected bound holds.\\
    
    For $\bm k'  = (k_1, \dots, k_{n-2})\in \eN^{n-2}$, we let 
    \[ \widetilde{\Psi_L}(\bm k') := \prod_{\ell = 1}^{n-2} \chi_\ell(k_\ell) \]so that 
    \[ \Psi_L(k_1, \dots, k_{n-1}) = \widetilde{\Psi_L}(\bm k') \chi_{n-1} (k_{n-1}).\]For $\bm k' \in \eN^{n-2}$ and for any Dirichlet character $\chi $, we define
    \[ \boldsymbol{\mathfrak{S}}_\chi (y,\bm k',a;v):= \sum_{\substack{k \leqslant y }} {\chi(k)  \over k}\rho_{F,ak_1\cdots k_{n-2}}^-(k;v) \gcd(k_1 \cdots k_{n-2}k,a),\]where $\rho_{F,ak_1\cdots k_{n-2}}^-(k;v) $ is as in (\ref{def func}). In this setting, the quantity $\boldsymbol{\mathfrak{S}}_I'(y,a;v)$ is equal to

    \[ \sum_{\substack{\bm k' \in \eN^{n-2} \\  k_1 \cdots k_{n-2} \leqslant y^{n-2} \\ \forall i \in I \smallsetminus \{n-1\}, \; k_i > y }} \!\!\!\!\!\!\!\widetilde{\Psi_L}(\bm k') {\rho_{F,a}^-(k_1\cdots k_{n-2};v) \over k_1 \cdots k_{n-2}} \left( \boldsymbol{\mathfrak{S}}_{\chi_{n-1}}\left( {y^{n-1} \over k_1 \cdots k_{n-2}} ,\bm k',a;v\right) - \boldsymbol{\mathfrak{S}}_{\chi_{n-1}} (y,\bm k',a;v)\right) . \] Applying Corollary \ref{lemme central general}, with $\psi_L$ replaced by $\chi_{n-1}$, we get that $\boldsymbol{\mathfrak{S}}_{\chi_{n-1}} \left( {y^{n-1} \over k_1 \cdots k_{n-2}} ,\bm k',a;v\right) - \boldsymbol{\mathfrak{S}}_{\chi_{n-1}} (y,\bm k',a;v)$ is
    \[\ll_\varepsilon \left( a^{1 + \varepsilon}(k_1 \cdots k_{n-2} )^\varepsilon \left(  y^{ \varepsilon - {1 \over 4n}} + (k_1 \cdots k_{n-2})^{{1 \over 4n} - \varepsilon} y^{-{n-1 \over 4n} + \varepsilon(n-1)} \right) \right). \] Now, we neglect the conditions $k_i  > y $ and we use the trivial bound for $\rho_{F,a}^-(\cdot ; v)$ to write
    
    \[ \sum_{\substack{\bm k' \in \eN^{n-2}  \\ k_1 \cdots k_{n-2} \leqslant y^{n-2} \\ \forall i \in I \smallsetminus \{n-1\}, \; k_i > y }} {\rho_{F,a}^-(k_1 \cdots   k_{n-2};v) \over (k_1 \cdots k_{n-2})^{1 - \varepsilon}}  \ll_\varepsilon \sum_{k \leqslant y^{n-2}} k^{ - 1+\varepsilon/2} \ll_\varepsilon y^{ (n-2) \varepsilon \over 2},  \]and 
    \[  \sum_{\substack{\bm k' \in \eN^{n-2}\\ k_1 \cdots k_{n-2} \leqslant y^{n-2}\\ \forall i \in I \smallsetminus \{n-1\}, \; k_i > y }}  {\rho_{F,a}^- (k_1 \cdots   k_{n-2};v) \over (k_1 \cdots k_{n-2})^{1 - {1 \over 4n}}}   \ll_\varepsilon \sum_{k \leqslant y^{n-2}} k^{ {1 \over 4n} + \varepsilon -1 } \ll_\varepsilon y^{ {n-2 \over 4n}  + (n-2) \varepsilon} .\]Therefore, since $|\widetilde{\Psi_L}(\bm k)| \leqslant 1$, we have $\boldsymbol{\mathfrak{S}}_I'(y,a;v) \ll_\varepsilon y^{(2n-3) \varepsilon - {1 \over 4n}} $ from which it follows that

 \begin{align*} & \sum_{\substack{ \bm k \in \eN^{n-1} \\ \bm k \notin [1,y]^{n-1}\\ 
    k_1 \cdots k_{n-1} \leqslant y^{n-1}}}\Psi_L(\bm k) {\rho_{F,a}^-(k_1\cdots k_{n-1};v) \over k_1 \cdots k_{n-1} }  \ll_\varepsilon  y^{(2n-3) \varepsilon - {1 \over 4n}},\end{align*}
    thus concluding the proof of the lemma.
\end{proof}

For technical reasons, we will need a bound for sums of the form

\[\label{H tilde} \tag{2.18}\boldsymbol{\mathfrak{S}}^{\mathrm{err}}(B,a;z;v;J):= \!\!\!\!\!\!\!\!\!\!\!\!\! \sum_{\substack{ \bm k \in \eN^{n-1} \\ k_1\cdots k_{n-1} \leqslant b_FB^2/z  \\
   \forall j \in J, \; k_j > z}} \!\!\!\!\!\!\!\!\!\!\! \Psi_L(\bm k){\rho_{F,a}^-(k_1\cdots k_{n-1};v) \over k_1 \cdots k_{n-1}} \gcd(k_1\dots k_{n-1},a ) \mathrm{vol}( \mathcal{R}(B,zk_1 \cdots  k_{n-1} )),\]where $v \in \mathcal{U}$, $z > 0$, $B \geqslant 2$ and $J  \subset \{ 1, \dots, n-1\}$.
   
\begin{lemme}\label{estim cal E}
    Let $L/\eQ$ be an abelian number field of degree $n \geqslant 3$, $F \in \eZ[s,t]$ be a binary quadratic form which is irreducible over $L$, let $W$ and $w_0$ be as in $($\ref{choice W}$)$, $I~\subset~\{~1, \dots , n-~1~\} $ with $\#I \geqslant 2$ and $v \in \mathcal{U}$. Let $\varepsilon \in \left( 0, {1 \over 4n} \right)$. Enlarging $w_0$ if necessary, for all $a \in \eN$ coprime to $W$, $B \geqslant 2$ and $z > 0$ such that $ z^2 \leqslant b_F B^2$, we have
     \[ \boldsymbol{\mathfrak{S}}^{\mathrm{err}}(B,a;z;v;I \smallsetminus \{i\}) \ll_\varepsilon a^{1+\varepsilon} {B^{2 + \varepsilon} \over z^{{ 1 \over 4n} - \varepsilon} }, \]for any fixed $i \in I$.
\end{lemme}
\begin{proof}
  For simplicity's sake, we start by the case $n = 3$. We have for instance
   \[ \boldsymbol{\mathfrak{S}}^{\mathrm{err}}(B,a;z;v; \{2\})= \sum_{\substack{ \bm k \in \eN^2 \\ k_1k_2 \leqslant b_FB^2/z  \\
   k_2 >z}} \!\!\!\! \Psi_L(\bm k){\rho_{F,a}^-(k_1k_2;v) \over k_1 k_2} \gcd(k_1k_2,a)\mathrm{vol}( \mathcal{R}(B,z k_1 k_2 )),\]where 
   \[ \Psi_L(\bm k) = \chi_1(k_1) \chi_2(k_2).\]For $\chi \in \widehat{G}$, $k_1,a \in \eN$ and $y > 0$, we introduce
   \[ \mathfrak{S}_\chi^{\mathrm{vol}}(y,B,a,k_1;z;v): = \sum_{\substack{  k \leqslant y}}  {\chi(k) \over k} \rho_{F,k_1a}^-(k;v) \gcd(k_1k,a) \mathrm{vol}( \mathcal{R}(B, zk_1 k )). \]Let $y_{k_1} := \left\lfloor b_FB^2/(z k_1)\right\rfloor$. We have that $\boldsymbol{\mathfrak{S}}^{\mathrm{err}}(B,a;z;v; \{2\})$ is equal to
   \[   \sum_{\substack{   k_1 \leqslant b_F B^2/z^2 }} \chi_1(k_1) {\rho_{F,a}^-(k_1;v) \over k_1} \left(\mathfrak{S}_{\chi_2}^{\mathrm{vol}}(y_{k_1},B,a,k_1;z;v) - \mathfrak{S}_{\chi_2}^{\mathrm{vol}}(z,B,a,k_1;z;v)\right). \]Note that $k_1$ goes to $b_F B^2/z^2$ to ensure that $z < k_2$ and $k_1k_2 \leqslant b_FB^2 / z$. Therefore, when replacing $\psi_L$ by $\chi_2$, Lemma \ref{estim final avec vol} provides that the quantity $\mathfrak{S}_{\chi_2}^{\mathrm{vol}}(y_{k_1},B,a,k_1;z;v) - \mathfrak{S}_{\chi_2}^{\mathrm{vol}}(z,B,a,k_1;z;v) $ is
   \begin{align*} \ll_\varepsilon   a^{1+\varepsilon}k_1^\varepsilon \left( B(k_1z)^{1/2} + (B^2 + k_1 z y_{k_1} \log B) y_{k_1}^{ - {1 \over 4n } + \varepsilon } +  (B^2 + k_1z^2 \log B) z^{ - {1 \over 4n } + \varepsilon }   \right) .  \end{align*}Replacing $y_{k_1}$ by its definition, we get that $\mathfrak{S}_{\chi_2}^{\mathrm{vol}}(y_{k_1},B,a,k_1;z;v) - \mathfrak{S}_{\chi_2}^{\mathrm{vol}}(z,B,a,k_1;z;v) $ is 
   \[  \ll_\varepsilon   a^{1+\varepsilon} k_1^\varepsilon \left( B(k_1z)^{1/2} + (k_1z)^{{1 \over 4n}- \varepsilon} B^{2- {1 \over 2n} + \varepsilon}  +  (B^2 + k_1z^2 \log B) z^{ - {1 \over 4n } + \varepsilon }   \right).  \]Hence, since $\rho_{F,a}^-(k_1;v) \ll_\varepsilon k_1^\varepsilon$, summing over $k_1 \leqslant b_FB^2/z^2$ provides that $\boldsymbol{\mathfrak{S}}^{\mathrm{err}}(B,a;z;v; \{2\})$ is
  \begin{align*}  \ll_\varepsilon & a^{1+\varepsilon} \left( \left( {B \over z} \right)^{{1 \over 2} +2 \varepsilon} \!\!\!\!\! B z^{1/2} + \left( {B \over z} \right)^{{1 \over 2 n} + 2 \varepsilon} \!\!\!\!\! z^{{1 \over 4n}-\varepsilon } B^{2 - {1 \over 2n} + \varepsilon} \right. \\ & \left. + \left( {B \over z} \right)^{2\varepsilon} B^2 z^{- {1 \over 4n} + \varepsilon} + \left( {B \over z} \right)^{2+2 \varepsilon}\!\!\!\! z^{2 - {1 \over 4n} + \varepsilon} \log B \right).  \end{align*}yielding 
  \[ \boldsymbol{\mathfrak{S}}^{\mathrm{err}}(B,a;z;v; \{2\}) \ll_\varepsilon a^{1+\varepsilon} {B^{2 + \varepsilon} \over z^{{ 1 \over 4n} - \varepsilon} }.\]For $n > 3$, we choose $j \in I \smallsetminus \{i\}$ and we fix $\bm k' := (k_1, \dots k_{j-1}, k_{j+1} , \dots , k_{n-2})$. We let
  \[ P_j (\bm k') := k_1 \cdots k_{j-1} k_{j+1} \cdots k_{n-1},  \]and
    \[ \widetilde{\Psi_L}(\bm k',j) := \prod_{\ell \neq j} \chi_\ell(k_\ell), \]so that 
    \[ \Psi_L(k_1, \dots, k_{n-1}) = \widetilde{\Psi_L}(\bm k',j) \chi_j (k_j).\] Let $\mathcal{A}_i(B,z)$ be defined as 
    \[ \mathcal{A}_{i,j}(B,z) := \left\{ \bm k' \in \eN^{n-2} : \begin{tabular}{c} $P_j(\bm k') \leqslant b_F (B/z)^2$, \\ 
    $\forall j' \in I \smallsetminus \{i,j\}$, $k_{j'} > z$  \end{tabular}  \right\}. \]It follows that $\boldsymbol{\mathfrak{S}}^{\mathrm{err}}(B,z,a;v;I \smallsetminus \{i\})$ is equal to
  \[  \sum_{\substack{ \bm k' \in \mathcal{A}_{i,j}(B,z) }}\!\!\!\!\!\!\widetilde{\Psi_L}(\bm k',j) {\rho_{F,a}^-(P_j(\bm k');v) \over P_j(\bm k') } \left( \mathfrak{S}_{\chi_j}^{\mathrm{vol}}(y_{P_j(\bm k')},B,a,P_j(\bm k');z;v) - \mathfrak{S}_{\chi_j}^{\mathrm{vol}}(z,B,a,P_j(\bm k');z;v) \right) .\]The conclusion follows as in the case $n = 3$, ignoring the condition $\forall j' \in I \smallsetminus \{i,j\}$, $k_{j'} > z$ in the set $\mathcal{A}_i(B,z)$.
   
\end{proof}

\section{Lower bound II : average of $r_L$ over the values of $F$}\label{section4}

Our goal is to find a sharp lower bound for the quantity $N_{F,L}(B)$ defined by (\ref{NFC0}), when $L/ \eQ$ is abelian of degree $n \geqslant 2$ and $F \in \eZ[s,t]$ is a binary quadratic form which is irreducible over $\eQ$. Recall that using (\ref{corps de classe}), we can identify any irreducible representation $\chi \in \widehat{G}$ with a non-trivial Dirichlet character modulo $q_L \in \eN$, still denoted $\chi$, where $q_L$ is the conductor of $L/\eQ$. Note that the value of $\chi(n)$ only depends on $n$ modulo $q_L$. The following result does not require $F$ to be irreducible over $L$.

\begin{prop}\label{first lower-bound} Let $L/\eQ$ be an abelian number field of degree $n \geqslant 2$, $F \in \eZ[s,t]$ be a binary quadratic form which is irreducible over $\eQ$, and let $W$ be as in $($\ref{choice W}$)$. Assume that there exist $x_0 \in L$ and $(s_0,t_0) \in \eZ^2$ such that $N_{L/\eQ}(x_0) = F(s_0,t_0)$ and the integer $F(s_0,t_0)$ is invertible modulo~$q_L$. There exist two integers $s_1$ and $t_1$ such that \begin{enumerate}
    \item[(i)] the integer $F(s_1,t_1)$ is invertible modulo $W$;
    \item[(ii)] $\forall \chi \in \widehat{G}, \; \chi(F(s_1,t_1)) = 1$;
\end{enumerate} and such that for all $B \geqslant 2$ and $z = B^\eta$ with $\eta \in (0,1)$, we have the estimate
 \[  N_{F,L}(B) \gg_\eta  \sum_{ \substack{(s,t) \in (\eZ\cap [-B,B])^2 \\ 
 \gcd(s,t) = 1 \\
 (s,t) \equiv (s_1,t_1) \bmod W} } {\mu^2(F(s,t)) \over n^{ \omega(F(s,t) , z) }} r_L(F(s,t)).\]
\end{prop}
\begin{proof}
    The image of $F(s_0,t_0)$ in the idèle class group is in the kernel of the Artin reciprocity map (see (\ref{corps de classe})) and therefore $\chi(F(s_0,t_0))=1$ for any $\chi \in \widehat{G}$. Since $F(s_0,t_0)$ is invertible modulo $q_L$, the Chinese remainder theorem and the choice of $W$ as in (\ref{choice W}) ensure that there exists $(s_1, t_1) \in \eZ^2$ such that $(s_1,t_1) \equiv (s_0,t_0) \bmod q_L$ and $F(s_1,t_1)$ is invertible modulo $W$. In particular, $F(s_1,t_1) \equiv F(s_0,t_0) \bmod q_L$ so $\chi(F(s_1,t_1)) = 1$. Now we write 
    \[ N_{F,L}(B) \geqslant  \sum_{ \substack{(s,t) \in (\eZ\cap [-B,B])^2 \\ 
 \gcd(s,t) = 1 \\
 (s,t) \equiv (s_1,t_1) \bmod W} } \mu^2(F(s,t) ) \mathds{1}_{\mathcal{N}_L} (F(s,t)). \]Lemma \ref{detecteurs} ensures that 
 \begin{align*} \mu^2(F(s,t)) \mathds{1}_{\mathcal{N}_L} (F(s,t)) & = \mu^2 (F(s,t))\prod_{p \mid F(s,t)} {(1 + \chi(p) + \cdots + \chi^{n-1}(p)) \over n} \\ & = {\mu^2(F(s,t)) \over n^{\omega(F(s,t))}}r_L(F(s,t)) .\end{align*}To conclude, we use that for our choice $z = B^\eta$ and $(s,t) \in [-B,B]^2$, we have \[\# \{ p \mid F(s,t) : p > z \} \leqslant { \log (F(s,t)) \over \log z}\ll_\eta 1\]so that
 \[  {1 \over n^{\omega(F(s,t))}} = {1 \over n^{\omega(F(s,t),z)} } \times {1 \over n^{\# \{ p \mid F(s,t) : p > z\} }} \gg_\eta {1 \over n^{\omega(F(s,t),z)}}.\]
\end{proof}

For $d \in \eN$ coprime to $W$ and for $(s_1,t_1)$ as in Proposition \ref{first lower-bound}, we let
\[ \label{def M} \tag{3.1} M_d(B) =  \sum_{ \substack{(s,t) \in (\eZ\cap [-B,B])^2 \\ 
 \gcd(s,t) = 1 \\ 
 (s,t) \equiv (s_1,t_1) \bmod W \\
 d \mid F(s,t) } } \mu^2(F(s,t)) r_L(F(s,t)).\]Let $(\lambda_d^-)_{d \geqslant 1}$ be such that \begin{enumerate}
    \item[(i)] we have $\lambda_1^- = 1$ and $\displaystyle \sum_{d \mid k} \lambda_d^- \leqslant 0$ for all $k > 1$,
    \item[(ii)] there exists $y >  0$ such that $\lambda_d^- = 0$ whenever $d > y$.
\end{enumerate}
The sequence $(\lambda_d^-)$ corresponds to a truncated version of the Möbius function, enabling to have a small support in $[1,y]$. An explicit expression for $\lambda_d^-$ can be found in \cite[§6]{IK}.

\begin{prop}\label{mise en place sieve}
 Let $L/\eQ$ be an abelian number field of degree $n \geqslant 2$, $F \in \eZ[s,t]$ be a binary quadratic form which is irreducible over $\eQ$, and let $W$ be as in $($\ref{choice W}$)$. For $B \geqslant 2$ and $z = B^\eta$ with $\eta \in (0,1)$, and for $\varepsilon_0> 0$, we have 
    \[\sum_{ \substack{(s,t) \in (\eZ\cap [-B,B])^2 \\ 
 \gcd(s,t) = 1 \\
 (s,t) \equiv (s_1,t_1) \bmod W} } {\mu^2(F(s,t)) \over n^{ \omega(F(s,t) , z) }} \geqslant\sum_{\substack{d \leqslant B^\varepsilon \\ p \mid d \Rightarrow p \leqslant z \\ \gcd(d,W) = 1}} \lambda_d^- \left(1- {1 \over n} \right)^{\omega(d)} M_d(B),\]where $(\lambda_d^-)$ is as above with parameter $y = B^{\varepsilon_0}$.
\end{prop}
\begin{proof}
We apply \cite[lemma 4.1]{Sofos} for a single multiplicative function $f$ defined by 
\[ f(p^m) := \begin{cases}
    {1 \over n} \text{ if } p \leqslant z \\
    1 \text{ otherwise}
\end{cases}\]for all prime $p$ and $m \in \eN$. The coprimality condition $\gcd(d, W)= 1$ is inherited from the fact that $F(s_1,t_1) $ is invertible modulo $W$.
\end{proof}

In section \ref{section5}, we will apply the fundamental lemma of sieve theory with the parameters $y = B^{\varepsilon_0}$ and $z = B^\eta$ for some $\eta < \varepsilon_0$ in $(0,1)$. Anticipating this application of the sieve, we set
\[ \label{choice eps0 eta} \tag{3.2} \varepsilon_0 := {1 \over 8 n^2}\quad  \text{ and } \quad \eta := {1 \over 16 n^2}. \]

We focus on estimating the quantity $M_d(B)$ when $d$ is square-free coprime to $W$. We define
\[ \label{def S_d}\tag{3.3} S_{d}(B,m):=\sum_{ \substack{(s,t) \in (\eZ\cap [-B,B])^2 \\ 
 \gcd(s,t) = 1 \\
 (s,t) \equiv (s_1,t_1) \bmod W \\ [d,m^2] \mid F(s,t)} } r_L(F(s,t))\]

\begin{prop}\label{lemme M_dY}
Let $\varepsilon > 0$. For all $B \geqslant 2$, $Y \in \eR \cap (1, B^{1/2})$ and $d$ square-free coprime to~$W$, we have \[ M_d(B) = \sum_{\substack{m \leqslant Y  \\ \gcd(m,W) = 1}} \mu(m)  S_{d}(B,m) + O_{\varepsilon} \left( {B^{2 + \varepsilon} \over Y} \right).\]
\end{prop}
\begin{proof}
    We use the identity
    
    \[ \mu^2(F(s,t)) = \!\!\!\!\! \sum_{\substack{m^2 \mid F(s,t)}} \!\!\! \mu(m) \]in order to write

    \[ M_d(B) = \sum_{\substack{m \in F([-B,B]^2) \\ \gcd(m,W) = 1}} \mu(m)  S_{d}(B,m).   \]Now, we deal with the contribution coming from $m > Y$, for $Y \leqslant B^{1/2}$. Since by (\ref{r_L convolee}) we have the inequality
    \[ r_L(F(s,t)) \ll_\varepsilon B^\varepsilon,\] we can bound this contribution by 
    \[ R(B) := B^\varepsilon \!\!\!\!\!\!\!\!\!\! \sum_{Y < m \leqslant (b_F)^{1/2} B} \# \{ (s,t) \in (\eZ \cap [-B,B])^2 :  m^2  | F(s,t) \}.\]Then, we use that the set $ \{ (s,t) \in (\eZ \cap [-B,B])^2 :  m^2  | F(s,t) \}$ is contained in a union of lattices of discriminant $m^2$, whose number is $\ll_\varepsilon B^\varepsilon$. Therefore, we have 
    \[  \# \{ (s,t) \in (\eZ \cap [-B,B])^2 :  m^2  | F(s,t) \} \ll_\varepsilon B^{\varepsilon} \left( {B^2 \over m^2} + 1 \right), \]leading to 
    \[ R(B) \ll_\varepsilon {B^{2 + \varepsilon} \over Y} \]since $Y \leqslant B^{1/2}$. 
\end{proof}

We now provide an estimate for the quantity $S_d(B,m)$ defined by (\ref{def S_d}) when $m~\leqslant~Y~\leqslant~B^{1/2}$. To this end, we use a method inspired by \cite[§4]{Lartaux} and relying on the technical results from section~\ref{section3}. Applying these lemmas require $F$ to be irreducible over $L$. We start by introducing some notation.\\ 

Let $v_0 \in \mathcal{U}$ be the multiplicative function defined by 
\[\tag{3.4}\label{v_0} v_0(k) := \prod_{p \mid k} \left( 1 + {1 \over p} \right)^{-1}.\]For any bounded region $\mathcal{R} \subset \eR^2$, we introduce the set

\[ \tag{3.5} \label{def Lambda}\Lambda^* (\mathcal{R},k) :=  \left\{ (s,t) \in \eZ^2 \cap \mathcal{R} : \begin{tabular}{c}  
 $\gcd(s,t) = 1$ \\ 
 $(s,t) \equiv (s_1,t_1) \bmod W$ \\
 $k \mid F(s,t)$ \end{tabular} \right\},\]where $(s_1,t_1)$ are as in Proposition \ref{first lower-bound}. Estimates for the number of elements in $\Lambda^* (\mathcal{R}, k)$ are well-known since the pioneering work of Daniel \cite{Daniel}, which has inspired \cite[§5.3]{Sofos}. If $k$ is coprime to $W$, and if $(s,t) \in \eZ^2 $ is such that $\gcd(s,t) = 1$ and $F(s,t) =0$, we have $\gcd(k,t) = 1$. We deduce the equality \

\[ \left\{ (s,t)\in \eZ^2 : \begin{tabular}{c} $\gcd(s,t) = 1$ \\ $F(s,t) \equiv 0 \bmod k$  \end{tabular} \right\} = \bigsqcup_{\substack{\xi \bmod k \\ F(\xi, 1) \equiv 0 \bmod k}} \left\{ (s,t) \in \eZ^2 : \begin{tabular}{c} $\gcd(s,t) = 1$ \\ $s \equiv \xi t \bmod k$  \end{tabular} \right\}.  \]Therefore, if $\mathcal{R} = \mathcal{R}(B,z)$ is of the form (\ref{region}), then \cite[lemma 5.3 (2)]{Sofos} provides a constant $c' > 0$ depending on $W$ such that 
 \[\label{estim réseaux} \tag{3.6}\# \Lambda^* (\mathcal{R}(B,z), k)  = c' {\mathrm{vol(\mathcal{R}(B,z)) } \rho_F^-(k)v_0(k)\over k} + O\left( B \log B \sum_{ \substack{\xi \bmod k \\ F(\xi, 1) \equiv 0 \bmod k}}  { 1\over  \lambda_1(k,\xi)} \right),\]where $\lambda_1(k,\xi)$ denotes the first successive minimum of the lattice 
 \[ \{ (s,t) \in \eZ^2 : s \equiv \xi t \bmod k \}.\]Minkowski's theorem \cite[chap VIII. (12) and (13) p. 203]{Cassels} yields $\lambda_1(k,\xi) \ll k^{1/2}$, hence 
 \[ \label{bound successive min} \tag{3.7} \sum_{k \leqslant y} \sum_{\substack{\xi \bmod k \\ F(\xi, 1) \equiv 0 \bmod k}} {1 \over \lambda_1(k,\xi) }   \leqslant \sum_{\substack{ \bm v \in \eZ^2 \\ \bm v \neq 0  \\ ||\bm v || \leqslant y^{1/2}} } {\tau(|F(\bm v)|) \over || \bm v ||} \ll_\varepsilon  y^{{1 \over 2} + \varepsilon}\quad \quad (y > 0).\]Finally, we recall that $\sigma_{k}\left( a\right)$ is defined by (\ref{vartheta}), and for $d, m \in \eN$ we let 
 \[\label{beta(d,m)} \tag{3.8} \beta(d,m) := \rho_F^- \left( dm\right) v_0(dm) \gcd(d,m)  ,\]and
 \[\label{alpha(d,m)} \tag{3.9} \alpha(d,m) := \beta(d,m) \sigma_1\left( {dm^2 \over \gcd(d,m^2) } \right).  \]

\begin{prop}\label{key prop}
   Let $L/\eQ$ be an abelian number field of degree $n \geqslant 2$, $F \in \eZ[s,t]$ be a binary quadratic form which is irreducible over $L$, and let $W$ and $w_0$ be as in~$($\ref{choice W}$)$. Let $\varepsilon \in \left( 0, {1 \over n^2} \right)$. Enlarging $w_0$ if necessary, for $B \geqslant 2$, $m \leqslant B^{1/2} $, $d \leqslant B^{{1 \over 8n^2}}$ satisfying $\mu^2(d) = \mu^2(m) = \gcd(dm,W)=1$, the quantity $S_d (B, m) $ is equal to

\begin{align*} c_0 u_{L,F}(v_0)(dm)  {\alpha(d,m)  B^2 \over  dm^2 }+ O_\varepsilon\left( d^{1 + \varepsilon} m^\varepsilon   B^{ 2 - {1\over n^2} + \varepsilon} \right),
\end{align*}with $c_0 := n c' c_{F,L}(v_0) > 0$, where $\alpha(d,m)$ is given by (\ref{alpha(d,m)}), $c'$ comes from (\ref{estim réseaux}), and where $c_{L,F}(v_0) > 0$ and $u_{L,F}(v_0) \in \mathcal{U}$ are as in (\ref{c_{F,L}(v)}) and (\ref{u_{F,L}(v)}).

\end{prop}
\begin{proof}Combining equality (\ref{r_L convolee}) with (\ref{def S_d}) provides 
\[ S_d(B,m)=\!\!\!\!\!\! \sum_{\substack{\bm k \in \eN^{n-1} \\ \gcd(k_1\cdots k_{n-1},  W) = 1}} \!\!\!\!\!\!\!\! \Psi_L(\bm k) \# \Lambda^*([-B,B]^2,[k_1 \dots k_{n-1} , d, m^2 ]),   \]and (\ref{estim réseaux}) provides an estimate for $\# \Lambda^*([-B,B]^2,[k_1 \dots k_{n-1} , d, m^2])$. We now isolate the dependence on $\bm k$ by writing 
\[ [k_1 \cdots k_{n-1}, d, m^2] = { k_1 \cdots k_{n-1} [d, m^2] \over \gcd( k_1\cdots k_{n-1} , [d,m^2] )} =  { k_1 \cdots k_{n-1} dm^2\over \gcd(d,m) \gcd \left( k_1\cdots k_{n-1} , {dm^2 \over \gcd(d,m)}\right)},   \]where we simplified $\gcd(d,m^2) = \gcd(d,m)$ using $\mu^2(d) = 1$. Using the fact that $v_0 \in \mathcal{U}$, we also write 
\[ v_0([k_1 \cdots k_{n-1}, d,m^2]) = v_0(dm) v_0(k_1 \cdots k_{n-1},dm).\]Moreover, since $\gcd(dm,W) = 1$ with $W$ such that $\rho_F^-(p^\nu) = \rho_F^-(p)$ for all $\nu \geqslant 1$ and for all $p \nmid W$, we write 

\[ \rho_F^-([k_1 \cdots k_{n-1},d,m^2]) = \rho_F^- \left( dm \right)  \rho_F^- \left( k_1 \cdots k_{n-1}, dm \right). \]Therefore, (\ref{estim réseaux}) implies that $ \# \Lambda^*([-B,B]^2,[k_1 \dots k_{n-1} , d, m^2])$ is equal to

\begin{align*}   4c'{\beta(d,m) \over dm^2} &{ \rho_{F,dm}^-(k_1\cdots k_{n-1};v_0) \over k_1 \cdots k_{n-1}} \gcd\left( k_1\cdots k_{n-1} , {d m^2 \over\gcd(d,m)} \right) B^2 \\
& + O_\varepsilon \left( B \log B \!\!\!\!\!\!\!\!\! \sum_{\substack{\xi \bmod k_1 \cdots k_{n-1} \\ F(\xi, 1) \equiv 0 \bmod k_1 \cdots k_{n-1}}} \!\!\!\!\!\! {1 \over \lambda_1(k_1\cdots k_{n-1},\xi)} \right), \end{align*}where $\rho_{F,a}^-(\cdot \; ; v_0)$ is defined by (\ref{def func}), $\beta(d,m)$ is given by (\ref{beta(d,m)}), and where we used the inequality $\lambda_1([k_1\cdots k_{n-1},d,m^2],\xi) \geqslant \lambda_1(k_1\cdots k_{n-1},\xi)$ which can be inferred from $k_1 \cdots k_{n-1} \mid [k_1 \cdots k_{n-1}, d,m^2]$. To deal with the error term and gain a small power of~$B$, we cut the sum over $\bm k$ into smaller contributions and we apply a symmetry argument inspired by Lartaux \cite[§4]{Lartaux}. We introduce the parameter
\[ z := B^{2/n}\]and we write
\[ \label{cut s} \tag{3.10}S_d (B, m)  = S_d^{(1)} (B, m)  + S_d^{(2)} (B, m), \]where
\[ S_d^{(1)} (B, m) :=  \sum_{\substack{\bm k \in [1,z]^{n-1} \\ \gcd(k_1\cdots k_{n-1},  W) = 1}} \!\!\!\!\!\!\!\! \Psi_L(\bm k) \# \Lambda^*([-B,B]^2,[k_1 \dots k_{n-1} , d, m^2 ])  \]and 

\[ S_d^{(2)} (B, m) :=  \sum_{\substack{\bm k \notin [1,z]^{n-1} \\ \gcd(k_1\cdots k_{n-1},  W) = 1}} \!\!\!\!\!\!\!\! \Psi_L(\bm k) \# \Lambda^*([-B,B]^2,[k_1 \dots k_{n-1} , d, m^2 ]).  \]Applying (\ref{estim réseaux}) and (\ref{bound successive min}) provides a constant $c' > 0$ depending on $W$ such that $S_d^{(1)}(B,m)$ is equal to
    \begin{align*} & c' { \beta(d,m) B^2 \over dm^2 } \boldsymbol{\mathfrak{S}}^\# \left(  [1,z]^{n-1}, {dm^2 \over \gcd(d,m)};  v\right)  + \; O_\varepsilon \left(   z^{{n- 1 \over 2} + \varepsilon} B \log B \right),
\end{align*}where $\mathfrak{S}^\# $ is as in (\ref{frac S diese}). Now, our choice of $z$ ensures that the error term obtained above is $\ll_\varepsilon B^{ 2 - {1\over n} + \varepsilon}$ which is admissible in view of the result we are aiming for. Corollary \ref{lemme central general} and Lemma~\ref{Bound error multi D 1}, when applied to $v = v_0 \in \mathcal{U}$, finally yield
\begin{align*} \label{estim s1} \tag{3.11}S_d^{(1)} (B,m)= \; &  c' c_{F,L}(v_0) u_{F,L}(v_0)(dm)   {\alpha(d,m) B^2 \over  dm^2 }  + O_\varepsilon\left( d^{1 + \varepsilon} m^\varepsilon B^{ 2 - {1\over n^2} + \varepsilon} \right) , 
\end{align*}where $c' > 0$, and $c_{F,L}(v_0)$, $u_{F,L}(v_0)$ are as in (\ref{c_{F,L}(v)}) and (\ref{u_{F,L}(v)}), and $\sigma_1$ is defined by (\ref{vartheta}).\\

We now deal with $S_d^{(2)}(B,m)$. Using the inclusion-exclusion principle as in the proof of Lemma \ref{Bound error multi D 1}, we write
\begin{align*} \tag{3.12} \label{cut s2}
    S_d^{(2)}(B,m) = \sum_{i = 1}^{n-1} S_{d,i}^{(2)} (B,m)   + O \left( \sum_{\substack{I \subseteq \{1, \dots, n-1\} \\ \# I \geqslant 2  }} \left| \mathcal{E}_I(B) \right| \right),
\end{align*}
where 
\[ \label{s2i} \tag{3.13}S_{d,i}^{(2)} (B,m) : =\!\!\!\!\!\!\!\! \sum_{\substack{ \bm k \in \eN^{n-1} \\ z < k_i \\ \gcd(k_1\cdots k_{n-1},  W) = 1}}  \!\!\!\!\!\!\!\!\!\!\!\! \Psi_L(\bm k) \# \Lambda^*([-B,B]^2,[k_1 \dots k_{n-1} , d, m^2 ])\]and
\[ \mathcal{E}_I(B) := \sum_{\substack{ \bm k \in \eN^{n-1} \\
   \forall i \in I, \; k_i > z\\ \gcd(k_1\cdots k_{n-1}, W)= 1}}  \!\!\!\!\!\!\!\!\!\!\!\! \Psi_L(\bm k) \# \Lambda^*([-B,B]^2,[k_1 \dots k_{n-1} , d, m^2 ]). \]Note that if $n = 2$ the sum over $I$ is empty. Since $(s,t) \equiv (s_1,t_1) \bmod W$ with $\chi(F(s_1,t_1)) =1 $ for all $\chi \in \widehat{G}$ and $F(s_1,t_1) \in (\eZ / W \eZ)^\times$ (see Proposition \ref{first lower-bound}), we can make the following change of variables. In the main term, for each $i \in \{ 1, \dots, n-1\}$, we let $k_i' = {F(s,t) \over  k_1\cdots k_{n-1}  }$. The point is that the equality $\chi(F(s,t)) = 1 $ for all $\chi \in \widehat{G}$ ensures that

\[ \label{sym1} \tag{3.14}\Psi_L(\bm k) = \chi_i^{-1}(k_i') \left( \prod_{\ell \neq i} \chi_i^{-1}(k_\ell) \chi_\ell(k_\ell) \right).\]The multiplication by $\chi_i^{-1}$ is a bijection from $ \widehat{G} \smallsetminus \{ \chi_i \}$ to $ \widehat{G} \smallsetminus \{1 \} $, so after a permutation of $k_1 , \dots, k_i', \dots, k_{n-1}$ we recover 
\[\label{sym2} \tag{3.15} \sum_{k_1\cdots k_i' \cdots k_{n-1} = k } \chi_i^{-1}(k_i') \left( \prod_{\ell \neq i} \chi_i^{-1}(k_\ell) \chi_\ell(k_\ell) \right) = \Psi_L(\bm k). \]This change of variables also replaces the quantity $\# \Lambda^*([-B,B]^2,[k_1 \dots k_{n-1} , d, m^2 ])$ with\[ \# \Lambda^*(\mathcal{R}(B, z k_1 \cdots k_i' \cdots k_{n-1}),[k_1 \cdots k_i' \cdots k_{n-1} , d, m^2 ]), \]for which estimate (\ref{estim réseaux}) still applies. Thus, if $b_F = \!\!\!\!\! \underset{(s,t) \in [-1,1]^2}{\max} |F(s,t)|$, letting $k = k_1 \cdots k_i' \cdots k_{n-1}$ yields
\begin{align*}
     S_{d,i}^{(2)}(B,m)  = \sum_{\substack{k \leqslant  b_F B^2/z \\ \gcd(k,  W) = 1}} \psi_L(k) \# \Lambda^*(\mathcal{R}(B,zk),  [k ,d,m^2]),
\end{align*}
where $\psi_L$ is defined in (\ref{psi}). We are now in a position to apply estimates (\ref{estim réseaux}) and (\ref{bound successive min}) to write that each $S_{d,i}^{(2)}(B,m)$ is equal to
 \begin{align*} & {c' \over 4} { \beta(d,m) \over  dm^2 }  \mathfrak{S}^{\mathrm{vol}}\left(b_F B^{2 - {2\over n}},{dm^2 \over \gcd(d,m)} ,1;z;v_0\right) +O_\varepsilon\left( B^{2 - {1 \over n}+ \varepsilon}  \right),
\end{align*}where $c' > 0$ is the same constant appearing for $S_d^{(1)}(B,d)$ and $\mathfrak{S}^\mathrm{vol}$ is as in (\ref{frac S vol}). We now apply Lemma \ref{estim final avec vol} with $z=B^{2 \over n}$. For all $i \in \{1, \dots, n-1\}$, we get
\begin{align*}\label{estim s2i} \tag{3.16}S_{d,i}^{(2)}(B,m)= \; &  c' c_{F,L}(v_0) u_{F,L}(v_0)(dm) {\alpha(d,m)B^2  \over  dm^2 }      + O_\varepsilon\left( d^{1+ \varepsilon}m^\varepsilon B^{ 2 - {1\over 4n} + \varepsilon} \right)  , 
\end{align*} 
where $c_{F,L}(v_0)$ and $u_{F,L}(v_0)$ are as in (\ref{c_{F,L}(v)}) and (\ref{c_{F,L}(v)}). For each error term $\mathcal{E}_I(B)$, we choose $i \in I$ and make the change of variable $k_i' = {F(s,t) \over k_1 \cdots k_{n-1}}$. Equalities (\ref{sym1}) and (\ref{sym2}) still hold and lead us to 
 \begin{align*} \mathcal{E}_I(B) = \sum_{\substack{ \bm k' \in \mathcal{A}_i (B,z)}}   \Psi_L(\bm k') \# \Lambda^*(  \mathcal{R}(B,z k_1 \cdots k_i'\cdots  k_{n-1} ), [k_1 \cdots k_i'\cdots  k_{n-1},d,m^2]) .\end{align*}
   where 
   \[ \mathcal{A}_i (B,z) = \left\{ \bm k' = (k_1, \dots, k_i', \dots, k_{n-1})  \in \eN^{n-1} : \begin{tabular}{c} $k_1\cdots k_i'\cdots k_{n-1} \leqslant b_FB^{2 - {2 \over n}}$  \\
   $\forall j \in I\smallsetminus \{i\}, \; k_j > z$ \\ $\gcd(k_1\cdots k_{n-1}, W)= 1$ \end{tabular} \right\} . \]Applying again estimate (\ref{estim réseaux}) for $k =[ k_1 \cdots k_i' \cdots k_{n-1},d,m^2]$, it follows that 
   \[ \mathcal{E}_I(B) \ll  \boldsymbol{\mathfrak{S}}^{\mathrm{err}} \left(B,dm;B^{2 \over n} ;v_0;I \smallsetminus \{i\} \right),\]where $\mathfrak{S}^\mathrm{err}$ is as in (\ref{H tilde}). We apply Lemma \ref{estim cal E} which provides
   \[ \label{estim ei} \tag{3.17}\mathcal{E}_I (B)  \ll_\varepsilon B^{2 - {1 \over n^2} + \varepsilon }.\]This concludes the proof of Proposition \ref{key prop}, up to combining equalities (\ref{cut s}), (\ref{estim s1}), (\ref{cut s2}), (\ref{estim s2i}) and (\ref{estim ei}).
\end{proof}

For any abelian extension $L/ \eQ$ and irreducible binary quadratic form $F \in \eZ[s,t]$, we let $g_{F,L}$ be the multiplicative function defined by 

\[\label{g1} \tag{3.18} g_{F,L}(d) := \rho_F^-(d)  \prod_{p \mid d}  \left(  1 + \psi_L(p) + \sum_{  \nu \geqslant 2 }  {\psi_L(p^\nu) \over p^\nu}    \right). \]

\begin{prop}\label{estim Md}  Let $L/\eQ$ be an abelian number field of degree $n \geqslant 2$, $F \in \eZ[s,t]$ be a binary quadratic form which is irreducible over $L$ and let $W$ and $w_0$ be as in $($\ref{choice W}$)$. Let $\varepsilon \in \left( 0, {1 \over 4n^2} \right)$. Enlarging $w_0$ if necessary, there exist $u \in \mathcal{U}$, and $c > 0$ such that for $B \geqslant 2$, $d \leqslant B^{{1 \over 8n^2}}$ satisfying $\mu^2(d)  = \gcd(d,W)=1$, we have
    \[ M_d(B) =c u(d) {g_{F,L}(d) \over d} B^2 + O_\varepsilon\left( d^{1 + \varepsilon} B^{2 - {1 \over 4n^2} + \varepsilon}  \right).\]
\end{prop}
\begin{proof}
    We have $\mu^2(d) = \mu^2(m) = 1$ so (\ref{vartheta}) becomes 
    \[ \sigma_1\left( {dm^2 \over \gcd(d,m^2) } \right) = \prod_{\substack{p \mid d \\ p \nmid m}} \left(1+ \sum_{  \nu \geqslant 1 }  {\psi_L(p^\nu) \over p^\nu}  \right) \prod_{ \substack{p \mid m }} \left(  1 + \psi_L(p) + \sum_{  \nu \geqslant 2 }  {\psi_L(p^\nu) \over p^\nu}    \right). \]We let 
    $u_1 \in \mathcal{U}$ be the function defined by 
    \[ u_1(k) := \prod_{p \mid k}\left(1+ \sum_{  \nu \geqslant 1 }  {\psi_L(p^\nu) \over p^\nu}  \right). \]Then, we have 
     \[ \rho_F^-(dm)  \sigma_1\left( {dm^2 \over \gcd(d,m^2) } \right) = { g_{F,L}(d)u_1(d) g_{F,L}(m,d)  \over  u_1(\gcd(m,d))},  \]with $g_{F,L}$ defined by (\ref{g1}). We now let 
     \[ u_2(d) := u_{F,L}(v_0)(d) u_1(d) \]and 
     \[ M_d'(Y) :=  \sum_{\substack{m \leqslant Y  \\ \gcd(m,W) = 1}}   { \mu(m)   \over m^2  }  { v_0(m,d)u_{F,L}(v_0)(m,d) \gcd(m,d ) g_{F,L}(m,d)  \over   u_1(\gcd(m,d)) }.   \]From Lemma \ref{lemme M_dY} and Proposition \ref{key prop}, for $Y \in (1, B^{1/2}) $, we get
    \[ M_d(B) =  c_0 {u_2(d)  g_{F,L}(d) \over d} B^2  M_d'(Y) + O_\varepsilon \left({B^{2 + \varepsilon} \over Y} + d^{1+\varepsilon} B^{ 2 - {1\over n^2} +  \varepsilon } Y^{1 + \varepsilon} \right) .\]Now, we write 
    \begin{align*}  \sum_{\substack{m > Y  \\ \gcd(m,W) = 1}}   { \mu(m)   \over m^2  }  { v_0(m,d)u_{F,L}(v_0)(m,d) g_{F,L}(m,d) \gcd(m,d) \over   u_1(\gcd(m,d))}  & \ll_\varepsilon d^\varepsilon \sum_{\substack{m > Y }}   { m^\varepsilon\over m^2  }  \gcd(d,m ) \\ & \ll_\varepsilon d^{1 + \varepsilon} Y^{\varepsilon - 1},  \end{align*}so that taking $Y = B^{1 \over 2n^2(1 + \varepsilon)} $ ensures that we can replace the sum over $M_d'(Y)$ by the complete sum 
    \[ \sum_{\substack{m \in \eN \\ \gcd(m,W) = 1}}  { \mu(m)   \over m^2  }  { v_0(m,d)u_{F,L}(v_0)(m,d) g_{F,L}(m,d)\gcd(m,d ) \over   u_1(\gcd(m,d)) },\]and the error term becomes 
    \[ \ll_\varepsilon B^{2 - {1 \over 2n^2(1 + \varepsilon)} + \varepsilon } + d^{1 + \varepsilon} B^{2 - {1-\varepsilon  \over 2n^2( 1+  \varepsilon)}  } + d^\varepsilon B^{ 2 - {1 \over 2n^2} +  \varepsilon  } \ll_\varepsilon d^{1+\varepsilon} B^{2- {1 \over 4n^2} + \varepsilon}. \] By multiplicativity, this sum over $m$ can be written as $c_{F,L}'  u_3 (d)$,
    where 
    \[ c_{F,L}' := \prod_{p \nmid W} \left( 1 - {1 \over p^2} v_0(p) u_{F,L}(v_0)(p) g_{F,L}(p)  \right), \]and
    \[ u_3 (d) :=  \prod_{p \mid d} \left( 1 - {u_1^{-1}(p)  \over p}  \right) \left( 1 - {1 \over p^2} v_0(p) u_{F,L}(v_0)(p) g_{F,L} (p) \right)^{-1}. \]
  Without loss of generality, we can choose $w_0$ in (\ref{choice W}) large enough to ensure that $c_{F,L}' > 0$ and $u_3 \in \mathcal{U}$, thus concluding the proof with $u := u_2u_3 = u_1 u_3 u_{F,L}(v_0)$.
    
\end{proof}

\section{Lower bound III : conclusion using the fundamental lemma of sieve theory}\label{section5}

Combining Propositions \ref{first lower-bound}, \ref{mise en place sieve} and \ref{estim Md}, we are in a position to conclude the proof of Theorem~\ref{th}. The last step relies on the following version of the fundamental lemma of sieve theory. For more details on this lemma, see \cite[th. 6.3]{IK}.

\begin{lemme}\label{sieve}
    Let $y > 1$ and $\kappa > 0 $. There exists two sequences of real numbers $(\lambda_d^\pm)_{d \geqslant 1}$ depending only on $y$, $\kappa$ such that 
    \begin{enumerate}
        \item[(i)] $\lambda_1^\pm = 1$,
        \item[(ii)] $|\lambda_d^\pm | \leqslant 1$ for all $d \geqslant 1$,
        \item[(iii)] $\lambda_d^\pm = 0$ if $d > y$,
        \item[(iv)] $\forall n \in \eN,\; \displaystyle \sum_{d \mid n} \lambda_d^- \leqslant 0 \leqslant \sum_{d \mid n} \lambda_d^+.$
    \end{enumerate}
    Moreover, if $g : \eN \to [0,1)$ is a multiplicative function satisfying 
    \[ \label{hyp sieve} \tag{4.1} \prod_{w \leqslant p < z} (1- g(p) )^{-1} \leqslant \left( { \log z \over \log w} \right)^\kappa \left( 1 + {M\over \log w} \right) \quad \quad \quad \quad (2 \leqslant w < z \leqslant y) \]where $M>0$ is independent of $w$, then for all real numbers $w,z,y$ satisfying $2 \leqslant w < z \leqslant y$, we have the estimates 
    \[ \tag{4.2}  \sum_{d \mid P(z) } \lambda_d^\pm g(d)  = \left(1 + O \left( \mathrm{e}^{- {\log y \over \log z}} \left( 1 + {M \over \log z}  \right)^{10}  \right) \right) \prod_{p < z} (1- g(p)) \]where the implicit constant only depends on $\kappa$.
\end{lemme}

Since $(\lambda_d^-)_{d \in \eN}$ has its support in $[1, y]$, choosing $y = B^{\varepsilon_0}$ as in Proposition \ref{mise en place sieve} enables us to write that the error term obtained when replacing $M_d(B)$ by its expression from Proposition~\ref{estim Md} is 
\[ \ll_\varepsilon B^{2 - {1 \over 4 n^2} + \varepsilon } B^{\varepsilon_0( 1 + \varepsilon)}. \]Hence, choosing $\varepsilon_0 = {1 \over 8n^2}$ as in (\ref{choice eps0 eta}) provides an error term
\[ \ll_\varepsilon B^{2 - {1 \over 8 n^2} + \varepsilon }, \]which is admissible. We now ensure that the multiplicative function $g$ defined by 
\[ g(d) := \mu^2(d) \left( 1-{1 \over n} \right)^{\omega(d)} u(d) { g_{F,L}  (d)  \over d}  \mathds{1}_{\gcd(d,W) = 1},\]where $g_{F,L}$ is as in (\ref{g1}) and $u \in \mathcal{U}$ satisfies the assumptions of Lemma \ref{sieve}. For any prime $p \nmid W$, we have

\[ g(p) = \left( 1 - {1 \over n}\right) \left({\rho_F^-(p) (1 + \psi_L(p)) \over p} + O \left( {1 \over p^2}  \right) \right)    \left( 1+ O \left( {1 \over p} \right) \right).\]Thus, we have
\[ \sum_{p <  z} g(p) =  \left(1 - {1 \over n} \right) \left( \sum_{ \substack{ p <  z\\ p \nmid W} } {\rho_F^-(p) \over p}+ \sum_{ \substack{ p < z \\ p \nmid W} } {\psi_L(p)\rho_F^-(p) \over p} \right)   + a + O\left({1 \over z} \right)\]with 
\[ a : = \sum_{k \geqslant 2} \sum_{p \nmid W} {g(p)^k \over k} + \sum_{p \nmid W} \left( g(p) - \left( 1 - {1 \over n} \right) {\rho_F^-(p)(1 + \psi_L(p)) \over p} \right).\]We now use the prime number theorem for the Dedekind zeta function of $K$ to write
\[ \sum_{ \substack{ p < z\\ p \nmid W} } {\rho_F^-(p) \over p} = \log_2 z + b_1 + O\left( {1 \over \log z} \right) \]and the prime number theorem for $L$-functions of non-trivial characters (see Lemma \ref{TNP}) provides
\[ \sum_{ \substack{ p < z\\ p \nmid W} } {\psi_L(p)\rho_F^-(p) \over p} = b_2 + O\left( {1 \over \log z} \right),\]where $b_1$ and $b_2$ are some constants depending only on $L,K$ and $W$. Hence, we are provided with the estimate
   \[   \log \prod_{p < z } (1 - g(p))^{-1}  = \left( 1- {1 \over n} \right) \log_2 z + a' + O \left({1 \over \log z} \right) \]with $a' = a + b_1 + b_2$. It follows that
   \[ \prod_{p < z } (1 - g(p))^{-1}  = \mathrm{e}^{a'} (\log z)^{1 - {1 \over n}} \left( 1 +  O \left({1 \over \log z} \right)\right).\]We recall that the integer $w_0$ is defined with $W$ in (\ref{choice W}). If $w > w_0$, upper bound (\ref{hyp sieve}) follows immediately. In the cases $w \leqslant w_0 < z$ and~$w~\leqslant~z~<~w_0$, the product $ \displaystyle \prod_{ w \leqslant p < z} (1 - g(p))^{-1}$ equals respectively $\displaystyle \prod_{w_0 + 1 \leqslant p < z} (1 - g(p))^{-1}$ and $1$ so that upper bound (\ref{hyp sieve}) is still satisfied. Therefore, we can apply Lemma \ref{sieve} to each function $g$. Taking $z = B^\eta$ with $\eta = {1 \over 16 n^2} < \varepsilon_0 $ as in (\ref{choice eps0 eta}), it follows from Lemma \ref{sieve} that 
\[ N_{F,L}(B) \gg {B^2 \over (\log B)^{ 1 - {1 \over n}}}. \]This concludes the proof of the lower bound in Theorem \ref{th}, in the case where $F$ is irreducible over~$L$.

\section{The case $F$ reducible over $L$}

In this section, we assume $F$ to be reducible over $L$ and $[L : \eQ] \geqslant 3$. By Lemma \ref{réécriture (iii)}, there exists a unique non-trivial character of $G$, denoted $\chi_0$, such that $\widetilde{\chi_0}$ is trivial. Since $K\subset L$, we have $2 \mid \# G$ and we can consider $M_1/ \eQ$, the maximal cyclic subfield of $L$ containing $K$ and such that $[M_1 : \eQ] $ is a power of $2$, say $2^a$ with $a \geqslant 1$. Let $H$ be the group satisfying $G \simeq \eZ/ 2^a \eZ \times H$. Then, the field $M_2 := L^H$ is such that $M_1 \cap M_2 = \eQ$, and there exists an isomorphism (see \cite[prop. 3.21]{Milne})
\[ G \overset{\phi}{\simeq}\mathrm{Gal}(M_1/ \eQ) \times \mathrm{Gal}(M_2/\eQ)\]given by 
\[ \phi : \sigma \longmapsto (\sigma_{\mid M_1} , \sigma_{\mid M_2} ).\]By duality, any character $\chi \in \widehat{G}$ can be written uniquely as $(\chi_1 \circ \pi_1 )(\chi_2 \circ \pi_2)$ with $\chi_i \in \widehat{\mathrm{Gal}(M_i / \eQ)} $ and with $\pi_i$ the projection
\[ \pi_i :  G  \longrightarrow \mathrm{Gal}(M_i / \eQ).\]If $p$ does not ramify in $L$, we have \[\phi(\mathrm{Frob}_p^L) = (\mathrm{Frob}_p^{M_1}, \mathrm{Frob}_p^{M_2}) = (\pi_1(\mathrm{Frob}_p^L), \pi_2(\mathrm{Frob}_p^L) ).\] Hence, for $p$ not ramifying in $L$, we can write
\begin{align*}  r_L(p) & = \sum_{\chi \in \widehat{G}} \chi(p) =  \sum_{\chi \in \widehat{G}} \chi(\mathrm{Frob}_p^L) \\
& = \sum_{\substack{\chi_1\in \widehat{\mathrm{Gal}(M_1/\eQ)} \\ \chi_2 \in \widehat{\mathrm{Gal}(M_2/\eQ)} }} \chi_1(\pi_1(\mathrm{Frob}_p^L))\chi_2(\pi_2(\mathrm{Frob}_p^L))\\
& =  \sum_{\substack{\chi_1\in \widehat{\mathrm{Gal}(M_1/\eQ)} \\ \chi_2 \in \widehat{\mathrm{Gal}(M_2/\eQ)} }} \chi_1(\mathrm{Frob}_p^{M_1})\chi_2(\mathrm{Frob}_p^{M_2})=   r_{M_1}(p ) r_{M_2}(p). \end{align*}Let $\chi_\ast$ be a generator of $\widehat{\mathrm{Gal}(M_1/ \eQ)} \simeq \eZ / 2^a \eZ$. Then $\chi_\ast^{2^{a -1}}$ is of order $2$ in $\widehat{\mathrm{Gal}(M_1 / \eQ)}$ so it is the restriction of $\chi_0$ to the subgroup $\mathrm{Gal}(M_1/\eQ)$ of $G$. Let $(s,t) \in \eZ^2$ be such that $F(s,t)$ is square-free and is the norm of an element of $L$. The ring $\mathcal{O}_L$ is assumed to be a principal ideal domain, so $F(s,t)$ is the norm of an ideal $\mathfrak{a} \in I_L$. If $p$ divides $F(s,t)$, then $p$ is the norm of a prime ideal $\mathfrak{p} \in \mathcal{P}_L$ and consequently, we have $\chi_\ast^{2^a -1} (p) = \chi_0(p) = \widetilde{\chi_0}(\mathfrak{p}) = 1 $.  We deduce that
\[ r_{M_1}(p ) = 2 (1 + \chi_*(p) + \cdots + \chi_*^{2^{a-1} -1}(p)). \]Now, the Galois correspondence ensures that there exists $M_1'/ \eQ$ such that \[ \mathrm{Gal}(M_1' / \eQ) \simeq \mathrm{Gal}(M_1/ \eQ) / \mathrm{Gal}(K/\eQ) \simeq \eZ / 2^{a-1} \eZ.\] In particular, we have the equality
\[ 1 + \chi_*(p) + \cdots + \chi_*^{2^{a-1} -1}(p) = r_{M_1'}(p).\]Taking $L_0$ the compositum of $M_1'$ and $M_2$ yields
\[ r_L(p) = 2 r_{M_1'}(p) r_{M_2}(p) = 2 r_{L_0}(p),\]where we used again that $\mathrm{Gal}(L_0 / \eQ) \simeq \mathrm{Gal}(M_1'/ \eQ) \times \mathrm{Gal}(M_2/ \eQ)$ to write $r_{L_0} (p)  =  r_{M_1'} (p ) r_{M_2}(p)$. Hence, Lemma \ref{first lower-bound} provides
\[  N_{F,L}(B) \gg_\eta  \sum_{ \substack{(s,t) \in (\eZ\cap [-B,B])^2 \\ 
 \gcd(s,t) = 1 \\
 (s,t) \equiv (s_1,t_1) \bmod W} } \mu^2(F(s,t)) \left( {2 \over n}\right)^{ \omega(F(s,t) , z) } r_{L_0}(F(s,t)).\]where $W$ is as in (\ref{choice W}) and $(s_1,t_1)$ is as in Lemma \ref{first lower-bound}. The form $F$ is irreducible over $L_0$ since by construction (following Lemma \ref{réécriture (iii)}) we have $\widetilde{\chi} $ non-trivial for every $\chi \in \widehat{\mathrm{Gal} (L_0/ \eQ)} $ which is non-trivial. We now apply Propositions \ref{mise en place sieve}, \ref{lemme M_dY}, \ref{key prop} and \ref{estim Md} with $r_{L_0}$ instead of $r_L$ in (\ref{def M}) and with $n$ replaced by $n/2$. Thus, we get 
 \[ N_{F,L}(B) \gg_\eta \sum_{\substack{d \leqslant B^\varepsilon \\ p \mid d \Rightarrow p \leqslant z \\ \gcd(d,W) = 1}} \lambda_d^- \left(1- {2 \over n} \right)^{\omega(d)} \left( c {g_1(d) \over d} B^2 + O_\varepsilon\left( d^{1 + \varepsilon} B^{2 - {1 \over 4n^2} + \varepsilon}  \right)  \right),\]where $g_1$ is as in (\ref{g1}) with $L_0$ instead of $L$. Finally, as in section \ref{section4}, Lemma \ref{sieve} provides
\[  N_{F,L}(B) \gg {B^2 \over (\log B)^{1 - {2 \over n}} },\]which concludes the proof of Theorem \ref{th}.
 
\section{Proof of Proposition \ref{coro}}\label{proof coro}

In this section, we prove Proposition \ref{coro}. This upper bound relies on a classical sieve for binary forms, first introduced by Nair in \cite{Nair}, then generalised by Nair and Tenenbaum in \cite{NT} for polynomials in one variable. A version of this result for irreducible binary forms has been developed by La Bretèche and Browning in \cite{BB}, before being generalised by Henriot in \cite{Henriot} and then by La Bretèche and Tenenbaum in \cite{BT}.\\

By \cite[lemma 8.2]{LM}, if the equation $N_{F,L}(x) = k$ has a solution everywhere locally, then for all $p$ dividing $k$ we have $[L_p : \eQ_p] \mid v_p(k)$. Thus, we deduce 
\[ N^\text{loc}_{F,L} (B) \leqslant \sum_{\substack{(s,t) \in \eN^2\\ |s|,|t| \leqslant B}} \varpi (F(s,t)),\]where $\varpi $ is the multiplicative function given by \[ \varpi(k) = \prod_{p \mid k} \mathds{1}_{[L_p : \eQ_p] \mid v_p(k)}. \]In particular, $\varpi(p) = \mathds{1}_{[L_p : \eQ_p] = 1}$ is the indicator function that $p$ splits completely in $\mathcal{O}_L$. We apply \cite[th. 1.1]{BT}, which yields

\[  N^\text{loc}_{F,L} (B)  \ll  B^2 \prod_{p \leqslant B} \left( 1 + {\rho_F(p) (\varpi(p) - 1) \over p^2}  \right).\]Hence, using again that $\rho_F(p) = 1 + (p-1)\rho_F^-(p)$, we have 
 \begin{align*} 
 { \rho_F(p)( \varpi(p)-1 ) \over p^2} & = \left( \varpi(p) - 1\right){\rho_F^-(p) \over p}  + O\left( {1 \over p^2} \right) . 
 \end{align*}Since $K/\eQ$ is Galois, any prime $p$ that does not ramify in $K$ satisfies
 \[ \rho_F^- (p) = \begin{cases}  \# \{ \mathfrak{p} \in \mathcal{P}_K : \mathfrak{p} \mid p \} = [K : \eQ] \text{ if } [K_p : \eQ_p] = 1 \\ 0 \text{otherwise.}\end{cases} \] Therefore, we have
 \[ \sum_{p \leqslant B} {\varpi(p) \rho_F^- (p) \over p} = [K: \eQ] \sum_{\substack{p \leqslant B \\ [L_p : \eQ_p] = 1 \\ [ K_p : \eQ_p] = 1 }} {1 \over p} + O(1).\]Since a prime that is unramified in both $L$ and $K$ splits completely in $K$ and $L$ if and only if it splits completely in their compositum $KL$ (see \cite[th. 3.1]{Marcus}), it follows that 
 \[ \sum_{p \leqslant B} {\varpi(p) \rho_F^- (p) \over p} = [K: \eQ]\sum_{\substack{p \leqslant B \\ [(LK)_p : \eQ_p] = 1 }} {1 \over p} + O(1). \]The Chebotarev density theorem \cite[th. 3.4]{Serre3} then yields
 \[ \sum_{p \leqslant B} {\varpi(p) \rho_F^- (p) \over p} = {[K : \eQ] \over [LK : \eQ]} \log_2(B) + O(1) . \]Finally, equality (\ref{calc [L: L cap K]}) ensures that 
 \[  {[K : \eQ] \over [LK : \eQ]}  = {r \over n}, \]where $r $ is the number of irreducible factors of $F$ in $L[s,t]$. It follows that
 
 \[ \sum_{p \leqslant B} {\rho_F(p)( \varpi(p)-1 ) \over p^2} =  - \left( 1 - {r \over n} \right) \log_2B + O(1),\]and
 \[ \prod_{p \leqslant B} \left(1+ { \rho_F(p)( \varpi(p)-1 ) \over p^2} \right) \ll (\log B)^{-\left(1 - {r\over n} \right)}. \]This concludes the proof of Proposition \ref{coro}.
 \qed
 
 \vspace{0.5cm}
\textbf{Acknowledgements :} I am thankful to Jean-Louis Colliot-Thélène for pointing out a mistake in an earlier version of this paper. I am also thankful to David Harari for useful conversations on norm equations, and to \'Etienne Fouvry for his valuable questions that helped to improve the content of this paper. I am grateful to Régis de la Bretèche for suggesting this problem and for his writing advices. I also thank Kevin Destagnol for his guidance throughout the writing of this paper.


\begin{thebibliography}{99}

\bibitem{BN}
T. Browning and R. Newton, The proportion of failures of the Hasse Norm principle. {\em Mathematika}, {\bf 62(2)} (2016), 337-347.


\bibitem{BB}
R. de la Bretèche and T. D. Browning. Sums of arithmetic functions over values of binary forms. {\em Acta Arith.}, {\bf 125(3)} : (2006) 291–304.

\bibitem{BT}
R. de la Bretèche and G. Tenenbaum, Moyennes de fonctions arithmétiques de formes binaires. {\em Mathematika}, {\bf 58(2)} : (2012) 290–304.

\bibitem{Cassels}
J.W.S. Cassels, An introduction to the Geometry of Numbers. Springer (1996).

\bibitem{CTHS}
J.-L. Colliot-Thélène, D. Harari and A.N. Skorobogatov, {\em Valeurs d’un polynôme à une variable représentées par une norme}, in “Number Theory and Algebraic Geometry”, ed. Miles Reid and Alexei Skorobogatov, London Mathematical Society Lecture Notes series {\bf 303} (2003), 69–89.

\bibitem{Daniel}
S. Daniel, On the divisor-sum problem for binary forms, { \em J. reine angew. Math.} {\bf 507} : (1999) 107–129.




\bibitem{Heilbronn}
H. Heilbronn, {\em Zeta-functions and L-functions}, in Algebraic Number Theory, 204–230.

\bibitem{Henriot}
K. Henriot, Nair-Tenenbaum bounds uniform with respect to the discriminant. {\em Mathematical Proceedings of the Cambridge Philosophical Society} ; {\bf 152} : (2012), no. 3, 405-424. 



\bibitem{IK}
H. Iwaniec and E. Kowalski, \textit{Analytic number theory}, American Mathematical Society Colloquium Publications, vol. 53, American Mathematical Society, Providence, RI, 2004.

\bibitem{Lang}
S. Lang, {\em Algebra}, Springer Science \& Business Media, 2012.

\bibitem{Lartaux}
A. Lartaux, Sur le nombre d'idéaux dont la norme est la valeur d'une forme binaire de degré 3, {\em  The Quart. J. Math.} {\bf 74.2} : (2023) 471-510.

\bibitem{LMFDB}
The LMFDB Collaboration, \textit{The number fields database, Home page of the abelian number fields with class number 1}, \url{https://www.lmfdb.org/NumberField/?field_is=ab&class_number=1}, (2026) [Online; accessed 21 February 2026].

\bibitem{LM}
D. Loughran and L. Matthiesen, 
Frobenian multiplicative functions and rational points in fibrations.  {\em J. Eur. Math. Soc.} {\bf 26} : (2024) 4779–4830.


\bibitem{LS}
D. Loughran and A. Smeets,
Fibrations with few rational points.
{\em GAFA} {\bf 26(5)} : (2016) 1449--1482.

\bibitem{Marcus}
D. Marcus, {\em Number Fields, Universitext}, Springer-Verlag (1977).

\bibitem{MV}
H. L. Montgomery and R. C. Vaughan, {\em Multiplicative number theory. I. Classical
theory}, Cambridge Studies in Advanced Mathematics, vol. {\bf 97}, Cambridge University
Press, Cambridge, 2007.

\bibitem{Milne}
J. S. Milne, Fields and Galois Theory, {\em Kea Books}, 2022.


\bibitem{Nair}
M. Nair. Multiplicative functions of polynomial values in short intervals. {\em Acta Arith.}, {\bf 62(3)} : (1992) 257–269.

\bibitem{NT}
M. Nair and G. Tenenbaum. Short sums of certain arithmetic functions. {\em Acta Math.}, {\bf 180(1)} : (1998) 119–144.

\bibitem{Neukirch}
J. Neukirch, Algebraic number theory. {\em Grundlehren der mathematischen Wissenschaften.} {\bf 322}, 1999.

\bibitem{Odoni}
R.W.K. Odoni, The Farey density of norm subgroups of global fields (I). {\em Mathematika}  {\bf 20(2)} : (1973) 155-169.


\bibitem{Serre3}
J.-P. Serre, 
\textit{Lectures on $N_X(p)$.} Aspects of Mathematics, {\bf 15}. 1997.


\bibitem{Serre1}
J.-P. Serre, 
Spécialisation des éléments de $\brau_2(\eQ(T_1 , \ldots, T_n))$.
{\em C. R. Acad. Sci. Paris Sér. I Math} {\bf 311} : (1990) 397-402.


\bibitem{Sofos}
E. Sofos, Serre's problem on the density of isotropic fibres in conic bundles. {\em Proc. London Math. Soc.} {\bf 113} : (2016).


\bibitem{Tate}
J. T. Tate, {\em Global Class Field Theory}, in Algebraic Number Theory, 162-203.

\bibitem{Tenenbaum}
G. Tenenbaum, {\em Introduction to analytic and probabilistic number theory}. Cambridge
University press, 1995.

\bibitem{Wei}
D. Wei, On the equation $N_{K/k}(\Xi) =f(t)$. {\em Proc. London Math. Soc.} {\bf 109} : (2014).
\end{thebibliography}
\end{document}